\documentclass[a4paper,11pt]{article}

\usepackage{amsmath,cases}
\usepackage{amssymb,amsfonts,bm}
\usepackage{amsthm}
\usepackage{dsfont}
\usepackage{color,graphicx,caption,subcaption}
\usepackage[svgnames]{xcolor}
\usepackage{hyperref}
\usepackage{xspace}
\usepackage{fullpage}
\usepackage{placeins}
\usepackage{enumitem}
\usepackage{epstopdf,caption,graphicx,subcaption,mathtools}
\usepackage{url}
\usepackage{verbatim}
\usepackage{amsthm}
\usepackage{commath}
\usepackage{multirow}
\usepackage[ruled]{algorithm2e}
\usepackage{ulem}
\usepackage{booktabs}
\usepackage{float}
\usepackage[makeroom]{cancel}
\usepackage{xpatch}
\xpatchcmd{\paragraph}{\normalfont}{{\normalfont\bfseries}}{}{}
\usepackage{tabularx}

\usepackage{hyperref}

\hypersetup{
  pdffitwindow=false,
  pdfhighlight=/O,
  pdfnewwindow,
  colorlinks=true,
  citecolor=red,             
  linkcolor=blue,            
  menucolor=blue,            
  urlcolor=blue,             
  pdfpagemode=UseOutlines,
  bookmarksnumbered=true,
  linktocpage,
  pdfkeywords={},
  pdfcreator={pdflatex},
  pdfproducer={LaTeX with hyperref}
}

\usepackage{geometry}
\newcommand{\ci}{\mathrm{i}} 

\definecolor{new-blue}{rgb}{0.0,0.0,0.8}

\newcommand{\be}{\begin{equation}}
\newcommand{\ee}{\end{equation}}
\newcommand{\ba}{\begin{array}}
\newcommand{\ea}{\end{array}}
\def\bea{\begin{eqnarray}}
\def\eea{\end{eqnarray}}
\def \beas{\begin{eqnarray*}}
\def \eeas{\end{eqnarray*}}
\newtheorem{definition}{Definition}[section]

\newtheorem{remark}{Remark}[section]

\newtheorem{thm}{Theorem}[section]
\newtheorem{lemma}{Lemma}[section]

\newtheorem{property}{Property}[section]

\DeclareMathOperator*{\argmin}{arg\,min}

\def\<{\langle}
\def\>{\rangle}

\def\ba{{\bf a}}

\usepackage{epsfig,epstopdf,bm}
\input psfig.sty

\begin{document}

\begin{center}
{\Large Structure and symmetry of the Gross--Pitaevskii ground-state manifold\renewcommand{\thefootnote}{\fnsymbol{footnote}}\setcounter{footnote}{0}
 \hspace{-3pt}\footnote{Zixu Feng and Qinglin Tang were partially supported by the National Key R\&D Program of China (Grant No.~2024YFA1012803) and the Natural Science Foundation of Sichuan Province (Grant No.~2024NSFSC0438). Patrick Henning was partially supported by the Deutsche Forschungsgemeinschaft (DFG, German Research Foundation; Grant No.~551527112).}}
\end{center}

\begin{center}
{\large Zixu Feng\footnote[1]{
School of Mathematical Science, Chengdu University of Technology, Chengdu 610059, P. R. China. (zixu\_feng123@163.com).}},
{\large Patrick Henning\footnote[2]{Department of Mathematics, Ruhr-University Bochum, DE-44801 Bochum, Germany. 
		  (patrick.henning@rub.de).}}, and
{\large Qinglin Tang\footnote[3]{School of Mathematics, Sichuan University, Chengdu 6100064, P. R. China. (qinglin\_tang@scu.edu.cn)}}\\[2em]
\end{center}

	\begin{abstract}
		The structure and degeneracy of ground states of the Gross--Pitaevskii energy functional play a central role in both analysis and computation, yet a precise characterization of the ground-state manifold in the presence of symmetries remains a fundamental challenge. In this paper, we establish sharp theoretical results describing the geometric structure of local minimizers and its implications for optimization algorithms. We show that when local minimizers are non-unique, the Morse--Bott condition provides a natural and sufficient criterion under which the ground-state set partitions into finitely many embedded submanifolds, each coinciding with an orbit generated by the intrinsic symmetries of the energy functional, namely phase shifts and spatial rotations. This yields a structural characterization of the ground-state manifold purely in terms of these natural symmetries.
		Building on this geometric insight, we characterize the local convergence behavior of general preconditioned Riemannian gradient methods (P-RG). Under the Morse--Bott condition, we derive sharp local $Q$-linear convergence estimates and prove that the condition holds if and only if the energy sequence generated by P-RG converges locally $Q$-linearly. In particular, on the ground-state set, the Morse--Bott condition is satisfied if and only if the minimizers decompose into finitely many symmetry orbits and the P-RG exhibits local linear convergence in a neighborhood of this set. When the condition fails, we establish a local sublinear convergence rate.
		Taken together, these results provide a complete and precise picture: for the Gross--Pitaevskii minimization problem, the Morse--Bott condition acts as the exact threshold separating linear from sublinear convergence, while simultaneously determining the symmetry-induced structure of the ground-state manifold. Our analysis thus connects geometric structure, symmetry, and convergence behavior in a unified framework.
	\end{abstract}
	
\vspace{6pt}


\vspace{6pt}
\normalsize
	
	\section{Introduction}
The Gross--Pitaevskii energy functional and the underlying equations form a central mathematical model in quantum physics. They were originally introduced to describe Bose--Einstein condensates (BECs), where a large number of bosonic particles occupy the same quantum state at extremely low temperatures. Due to their ability to capture collective quantum behavior, these models have found applications in several areas, including cold atom physics, nonlinear optics, astrophysical modeling, and the study of quantum fluids and turbulence \cite{1995Observation,2014Intro,2013Quan, K1995Bose,2000Fuzzy,Klaers2010Bose}. For example, related equations appear in nonlinear optics to describe light propagation in nonlinear media, while in astrophysics they are used in models where macroscopic quantum coherence is expected, such as ultra-light dark matter or superfluid phases inside neutron stars. Moreover, the Gross--Pitaevskii equation provides an important framework for investigating vortex dynamics and energy transfer processes in quantum turbulence.

Accordingly, the structure and characterization of minimizers of the Gross–Pitaevskii energy functional are of central importance, both for the mathematical analysis of Bose–Einstein condensates and related quantum systems, and for their reliable numerical computation. From a mathematical standpoint, these minimizers arise from a constrained variational problem under an $L^2$ normalization condition. Following the presentation in the survey by Bao et al.~\cite{2013Mathematical}, the dimensionless Gross--Pitaevskii energy functional in a rotating frame is defined by
\begin{align}\label{GP-Energy}
	E(\phi):=\frac{1}{2}\int_{\mathbb{R}^d}\left(\frac{1}{2}|\nabla\phi|^2+V(\bm{x})|\phi|^2
	-\Omega\overline{\phi} \mathcal{L}_z\phi+F(\rho_{\phi})\right)\text{d}\bm{x},
\end{align}
where $\phi$ denotes the macroscopic wave function describing the quantum state of the condensate. Furthermore, $\bm{x}\in\mathbb{R}^d\ (d=2,3)$ denotes the spatial variable, with $\bm{x}=(x,y)^{\top}$ in two dimensions and $\bm{x}=(x,y,z)^{\top}$ in three dimensions. The trapping potential $V(\bm{x})$ (that confines the particles) is real-valued and satisfies $\lim_{|\bm{x}|\to\infty}V(\bm{x})=\infty$. The rotation component is described by the angular momentum operator $\mathcal{L}_z=-\ci(x\partial_y-y\partial_x)$ together with the rotation frequency $\Omega\ge0$. The notation $\overline{\phi}$ denotes the complex conjugate of $\phi$. The nonlinear particle interaction term is given by
\begin{align}
	\label{F-density-notation}
	F(\rho_{\phi})=\int_{0}^{\rho_{\phi}}f(s)\;\text{d}s,\quad\ \rho_{\phi}:=|\phi|^2,
\end{align}
i.e., a function acting on the particle density $\rho_{\phi}$. The function $f(s)$ frequently appears in forms such as $f(s)=\eta s$, $\eta s\log s$, or $\eta s+\eta_{\mbox{\tiny LHY}} s^{3/2}$, depending on the physical applications, cf. \cite{2021Rota,2006Deri,2020Supp,2019Rota}. The normalization constraint is defined by
\begin{align*}
	N(\phi):=\|\phi\|_{L^2(\mathbb{R}^d)}^2=\int_{\mathbb{R}^d}|\phi|^2\;\text{d}\bm{x}=1
\end{align*}
and represents the normalization of the total particle number (mass). The corresponding ground state wave function $\phi_g$, in a given physical configuration, is therefore characterized by the constrained minimization problem
\begin{align}\label{O-P}
	\phi_g:=\argmin_{\phi\in\mathcal{M}} E(\phi)\quad
	\mbox{with} \quad  \mathcal{M}:=\left\{\phi\in H^1(\mathbb{R}^d) \,\,\big|\,\|\phi\|_{L^2(\mathbb{R}^d)}^2=1\right\}.
\end{align}
The main analytical challenges associated with problem \eqref{O-P} arise from the non-convex constraint and the inherent symmetry properties of the Gross--Pitaevskii functional.
A first symmetry originates from global phase invariance of the energy: If $\phi_g$ is a local minimizer, then $e^{\ci\alpha}\phi_g$ is also a local minimizer for every $\alpha \in [-\pi,\pi)$.
A second symmetry is induced by rotational invariance of the trapping potential: If $V(\bm{x})$ is radially symmetric with respect to the $z$-axis, i.e., $V(\bm{x})=V(A_{\beta}\bm{x})$ for all $\beta \in [-\pi,\pi)$, where
\begin{align*}
	A_{\beta}=\left(\begin{matrix}
		\cos\beta &-\sin\beta\\
		\sin\beta &\cos\beta
	\end{matrix}\right)\ \text{for}\; d=2,\quad A_{\beta}=\left(\begin{matrix}
		\cos\beta &-\sin\beta &0\\
		\sin\beta & \cos\beta &0\\
		0      & 0         &1
	\end{matrix}\right)\ \text{for}\; d=3,
\end{align*}
then $\phi_g(A_{\beta}\bm{x})$ is likewise a local minimizer. These continuous symmetry transformations imply that local minimizers are generally not isolated but instead form families of symmetry-related states, which significantly complicates both theoretical analysis and numerical computation.

Over the past two decades, various approaches have been proposed for computing minimizers of the Gross–Pitaevskii energy functional, including gradient-flow-based methods \cite{2026JCAM,2014Robust, 2017Efficient, 2003Computing, 2006Efficient, 2023Second, 2024On, 2000Ground, 2010A, 2017Computation, 2025OnBest, 2001Optimi, 2020Sobolev, 2023The, 2010Tackling, 2021Normalized, 2025WuLiuCai, 2017A, 2024Anew, 2022Exp, 2019Efficient} and nonlinear eigenvalue solvers \cite{2021The, 2007Dion, 2023The, 2014An}. From an analytical perspective, however, the presence of symmetry-induced degeneracy raises fundamental questions about the structure of the set of minimizers and its impact on optimization methods. In this work, we focus on preconditioned Riemannian gradient methods \cite{2024Riem, 2024On, 2010A, 2017Computation, 2025OnBest, 2020Sobolev, 2023The, 2024Convergence, 2010Tackling,2022Exp,2025Zhang} as a representative class and study how their local behavior is governed by the geometry of the critical set.

From a geometric viewpoint, several rigorous convergence results are available, but they rely on structural assumptions on the set of minimizers that reflect the symmetry-induced degeneracy described above. In particular, local linear convergence of Riemannian gradient methods toward the manifold of ground states generated by phase invariance was recently established in \cite{2024Convergence}, thereby exploiting the fact that the non-uniqueness of minimizers is locally induced by a single continuous symmetry. This perspective was further developed in \cite{2025Zhang}, where several Sobolev gradient variants were analyzed on an appropriate quotient space, again under the assumption that the ground-state manifold is locally described by symmetry orbits arising from phase invariance. More recently, \cite{2025OnBest} treated general preconditioners and obtained explicit 
 convergence estimates 
 by establishing a Polyak--\L ojasiewicz inequality under a structural assumption on the critical set which, in the Gross--Pitaevskii setting, corresponds to the ground-state manifold being locally generated by the natural symmetries of the model, namely phase shifts and possibly spatial rotations. This condition can be interpreted as a Morse--Bott-type assumption tied to symmetry-induced degeneracy.

Despite these advances, the current theory does not provide a complete geometric characterization of when linear convergence should be expected. Existing results show that linear convergence holds under symmetry-driven structural assumptions on the ground-state manifold, but they do not address the more general situation in which the critical set satisfies a Morse--Bott condition in an intrinsic sense, independent of a priori identification with specific symmetry orbits. In particular, it remains unclear whether linear convergence is fundamentally equivalent to such a geometric condition, how the structure of the set of minimizers is related to it, and what convergence behavior should occur when this structure fails. A unified picture linking the geometry of the critical set, the symmetry-induced classification of minimizers, and the precise convergence regime of preconditioned Riemannian gradient methods is therefore still missing.

A recent survey \cite{2025SIAMRev} highlights several open problems related to these phenomena, including whether symmetry-related families of nearly degenerate states can lead to arbitrarily slow convergence, how the entire set of ground states can be systematically classified, and which parameters govern or accelerate convergence in concrete applications. The results developed in this paper provide new insight into these questions by linking the geometric structure of the critical set to the local convergence behavior of Riemannian gradient methods. In particular, the results identify the Morse--Bott condition as the key geometric mechanism that governs both the structure of the ground-state set and the transition between linear and sublinear local convergence rates.
	
The remainder of the paper is organized as follows. Section \ref{Sec2} introduces notation, assumptions, and fundamental properties of the minimization problem together with the relevant aspects of preconditioned Riemannian optimization. In Section \ref{Sec3}, we analyze the geometric structure of the energy landscape and establish its connection with the convergence behavior of the P-RG iteration. 
Concluding remarks are given in Section \ref{Sec6}.
	
	\section{Preliminaries}
	\label{Sec2}
	In this section, we introduce the problem setting, basic notation, and some important properties of the problem and the general class of the Riemannian gradient methods.
	
	\subsection{Problem settings and notations}
	In our analytical framework, the physical domain is truncated from the full space $\mathbb{R}^d$ to the bounded domain $\mathcal{D}$ and the homogeneous Dirichlet boundary condition is imposed on $\partial\mathcal{D}$ due to the trapping potential. On $\mathcal{D}$, we adopt the standard notations for the Lebesgue spaces $L^p(\mathcal{D})=L^p(\mathcal{D},\mathbb{C})$ and the Sobolev space $H^1(\mathcal{D})=H^1(\mathcal{D},\mathbb{C})$ as well as the corresponding norms $\|\cdot\|_{L^p}$ and $\|\cdot\|_{H^1}$. For notational simplicity, we omit the explicit dependence on $\mathcal{D}$ in these norms. Recalling the notation $F(\rho_{\phi})$ from \eqref{F-density-notation}, we then consider the Gross--Pitaevskii energy functional \eqref{GP-Energy} and the constrained optimization problem \eqref{O-P} on $\mathcal{D}$:
	\begin{align}\label{Riem-Opt-Problem}
		\nonumber E(\phi)&:=\frac{1}{2}\int_{\mathcal{D}}\left(\frac{1}{2}|\nabla\phi|^2+V(\bm{x})|\phi|^2
		-\Omega\overline{\phi} \mathcal{L}_z\phi+F(\rho_{\phi})\right)\text{d}\bm{x}\\
		\text{and}\qquad\phi_g&:=\argmin_{\phi\in\mathcal{M}} E(\phi)\quad
		\mbox{with} \quad  \mathcal{M}:=\left\{\phi\in H_0^1(\mathcal{D})\big|\|\phi\|_{L^2}^2=1\right\}.
	\end{align}
	The set $\mathcal{M}$ forms a Riemannian manifold, whose tangent space at $\phi \in \mathcal{M}$ is given by
	\begin{align}\label{Tangent-space}
		T_{\phi}\mathcal{M}:=\left\{v\in H^1_0(\mathcal{D})\;\Big|\;\text{Re}\int_{\mathcal{D}}\phi\overline{v}\;\text{d}\bm{x}=0\right\}.
	\end{align}
	
	Since the Gross--Pitaevskii energy functional \(E\) is real-valued while the wave function \(\phi\) is complex-valued, \(E\) is not complex Fr\'echet differentiable in the usual sense. To address this, we work within a real-linear space consisting of complex-valued functions, as done in \cite{2021The, 2003Semi}. In this setting, the function space is viewed as a real Hilbert space, meaning that all variations are taken with respect to real parameters. To this end, we equip the Lebesgue space \(L^2(\mathcal{D})\) and the Sobolev space \(H^1_0(\mathcal{D})\) with the following real inner products:
	\begin{align*}
		(u,v)_{L^2}:=\text{Re}\int_{\mathcal{D}}u \overline{v}\;\text{d}\bm{x}\qquad \text{and}\qquad (u,v)_{H^1}:=\text{Re}\left(\int_{\mathcal{D}}u \overline{v}\;\text{d}\bm{x}+\int_{\mathcal{D}}\nabla u \cdot \overline{\nabla v}\;\text{d}\bm{x}\right).
	\end{align*}
	The corresponding real dual space is denoted
	by $H^{-1}(\mathcal{D}):=\big(H^1_0(\mathcal{D})\big)^*$. 
	For any $\phi \in H_0^1(\mathcal{D})$, let $\mathcal{P}_{\phi}: H_0^1(\mathcal{D}) \to H^{-1}(\mathcal{D})$ be a symmetric, coercive real-linear preconditioner. It induces a bilinear form $(\cdot,\cdot)_{\mathcal{P}_{\phi}} := \big\langle \mathcal{P}_{\phi}\cdot, \cdot \big\rangle$ where $\langle\cdot,\cdot\rangle$ represents the canonical duality pairing between $H^{-1}(\mathcal{D})$ and $H_0^1(\mathcal{D})$. This bilinear form induces an inner product on $H_0^1(\mathcal{D})$, with the associated norm given by $\|v\|_{\mathcal{P}_{\phi}} := \sqrt{\langle\mathcal{P}_{\phi}v, v\rangle}$. Furthermore, for any closed subset $W \subset T_{\phi}\mathcal{M}$, its orthogonal complement with respect to this inner product is
	\begin{equation}\label{Orth-Comp}
		W^{\bot}_{\mathcal{P}_{\phi}} := \left\{ u \in T_{\phi}\mathcal{M} \,\,\big|\, (u,v)_{\mathcal{P}_{\phi}} = 0 \,\ \forall v \in W \right\}.
	\end{equation}
	Given a subset $\mathcal{U}\subset\mathcal{M}$, its $\sigma$-neighborhood is defined as
	\begin{align}\label{sigma-neighborhood}
		\mathcal{B}_{\sigma}(\mathcal{U}):=\left\{\varphi\in\mathcal{M}\,\big|\,\exists\;\phi\in\mathcal{U}:\,\|\varphi-\phi\|_{H^1}<\sigma\right\}.
	\end{align}
	Throughout the paper, we use two types of constants: $(i)$ Generic constants denoted by $C$, which depend only on $\mathcal{D}$, $d$, $K$, and $V_\infty := \|V\|_{L^\infty}$; $(ii)$ Parameter-dependent constants written as $C_{v_1,\dots,v_k}$, which increase monotonically with the $H^1$-norms of the functions $v_1,\dots,v_k$. In particular, if $\|v_j\|_{H^1} \leq M$, then $C_{v_1,\dots,v_j,\dots,v_k} \leq C_{v_1,\dots,M,\dots,v_k}$.

	Throughout the remainder of the paper, we work under the following standing assumptions.
	\begin{itemize}
		\item[\bf (A1)] $\mathcal{D} \subset \mathbb{R}^d$ is a bounded domain with $C^{1,1}$ boundary, and is rotationally symmetric about the $z$-axis for $d=2,3$, such as a disk for $d=2$ and a ball for $d=3$.
		\item[\bf (A2)] $V\in L^{\infty}(\mathcal{D})$ is rotationally symmetric about the $z$-axis, i.e., $V(\bm{x})=V(A_{\beta}\bm{x})$ for all $\beta\in[-\pi,\pi)$. Moreover, the trapping potential dominates the centrifugal contribution, i.e., there exists a constant $K>0$ such that
		\begin{align*}
			V(\bm{x})-\frac{1+K}{2}\Omega^2(x^2+y^2)\ge0
			\quad \text{for a.e. }\bm{x}\in\mathcal{D}.
		\end{align*}
		\item[\bf (A3)] The nonlinearity $ f : [0,\infty) \to [0,\infty) $ satisfies
		$
		f \in C([0,\infty)) \cap C^1((0,\infty)), f(0) = 0,
		$
		and the limit $ \lim_{s \to 0^+} f'(s^2) s^2 = 0 $ exists. Furthermore, there exists $\theta\in\left[0,3\right)$ such that $f'(s^2)s^2$ is Lipschitz continuous with
		polynomial growth, i.e.,
		\begin{align*}
			\left|f'(s_1^2)s_1^2-f'(s_2^2)s_2^2\right|\le C\left(s_1+s_2\right)^{\theta}|s_1-s_2|,\quad \forall\; s_1, s_2\ge0.
		\end{align*}
		\item[\bf (A4)] Given $\phi\in H_0^1(\mathcal{D})$ and for all $u,v\in H_0^1(\mathcal{D})$,\, $\mathcal{P}_{\phi} : H^1_0(\mathcal{D}) \rightarrow H^{-1}(\mathcal{D})$ satisfies:
		\begin{itemize}
			\item[$(i)$] $\mathcal{P}_{\phi}$ is symmetric, coercive, and continuous on $H_0^1(\mathcal{D})$, i.e.,
			\begin{align*}
				\left\langle \mathcal{P}_{\phi} v, v \right\rangle \ge C\|v\|_{H^1}^2 \;\text{and}\;
				\left\langle \mathcal{P}_{\phi} u, v \right\rangle=	\left\langle \mathcal{P}_{\phi} v, u \right\rangle \le C_{\phi} \|u\|_{H^1} \|v\|_{H^1}.
			\end{align*}
			
			\item[$(ii)$] Given $\psi\in H_0^1(\mathcal{D})$, the following inequality holds
			\begin{align*}
				\left|\left\langle \big(\mathcal{P}_{\phi}-\mathcal{P}_{\psi}\big)u,v\right\rangle\right|\le C_{\phi,\psi}\|u\|_{H^1}\|v\|_{H^1}\|\phi-\psi\|_{H^1}.
			\end{align*}
		\end{itemize}
	\end{itemize}
	We briefly discuss the role and interpretation of the assumptions.
	
	Assumptions {\bf (A1)} and {\bf (A2)} guarantee rotational invariance of the Gross--Pitaevskii energy functional under rotations about the $z$-axis. 
	Due to the normalization constraint, the energy functional already possesses a basic invariance with respect to multiplication by complex phase factors, so that minimizers are never isolated in a strict sense. The additional rotational invariance introduced by {\bf (A1)}--{\bf (A2)} has important analytical consequences: rotating a minimizer produces another minimizer with the same energy, leading to higher-dimensional families of minimizers and additional degeneracies in the second variation of the energy.
	
	Under the assumed $C^{1,1}$ regularity and rotational symmetry of $\mathcal{D}$, these degeneracies admit a precise mathematical characterization. In particular, for a local minimizer $\phi_g$ one has $\mathcal{L}_z \phi_g \in H_0^1(\mathcal{D})$, and the function $\ci \mathcal{L}_z \phi_g$ represents an infinitesimal generator of the rotational symmetry. Consequently, $\ci \mathcal{L}_z \phi_g$ belongs to the tangent space $T_{\phi_g}\mathcal{M}$ and gives rise to a zero eigenfunction of the Riemannian Hessian of $E$ on $\mathcal{M}$ at $\phi_g$. A detailed justification of the inclusion $\mathcal{L}_z \phi_g \in H_0^1(\mathcal{D})$ under these geometric assumptions is provided in Appendix \ref{Appendix_Lz_phi_g}.
	
	We emphasize that the presence of rotational symmetry represents the most delicate setting for the analysis, since it enlarges the kernel of the second variation and introduces further degeneracies. When rotational invariance is absent, for instance due to the geometry of the domain or the external potential, these symmetry-induced degeneracies disappear, and the structure of the minimizer set simplifies accordingly.

	Even when $\mathcal{D}$ lacks rotational symmetry or $C^{1,1}$ regularity (e.g., rectangular domains), numerical evidence suggests that the symmetry-induced structure described above is still effectively observed, provided the computational domain is sufficiently large. 	
	This behavior can be explained by the exponential decay of the ground state $\phi_g$, which renders the numerical solution effectively insensitive to the boundary geometry. In this sense, there exists a rotationally symmetric subdomain $\widetilde{\mathcal{D}}\subset\mathcal{D}$ with $C^{1,1}$-boundary containing the essential support of $\phi_g$, such that $\phi_g$ coincides with a rotationally symmetric ground state on $\widetilde{\mathcal{D}}$ up to exponentially small errors. On this interior subdomain, the inclusion $\ci \mathcal{L}_z \phi_g \in H_0^1(\widetilde{\mathcal{D}})$ holds rigorously, and the corresponding zero-mode structure is therefore accurately captured in numerical computations on the full domain $\mathcal{D}$.
	
	Regarding {\bf (A3)}, we note that the condition $f\ge 0$ 
	can be relaxed to a lower boundedness condition; for clarity, we assume non-negativity. The growth and regularity condition on $f'$ is adapted from the classical work~\cite{2003Semi} and ensures that the energy functional $E$ is twice continuously Fr\'echet differentiable, i.e., $E\in C^2(H_0^1(\mathcal{D}),\mathbb{R})$.
	
	Finally, assumption {\bf (A4)} is not a structural property of the Gross--Pitaevskii model itself, but rather a condition linked to the numerical framework, namely to the preconditioner $\mathcal{P}_\phi$ employed in the Riemannian gradient method. 
	It ensures that $\mathcal{P}_\phi$ induces a stable local metric on $\mathcal{M}$. 
	Following~\cite{2025Other}, we omit the compactness assumption on $\mathcal{P}_\phi$ previously imposed in~\cite[Assumption~(A6)-(iii)]{2025OnBest}, as it is not essential for the analysis of local convergence rates. Condition {\bf (A4)} is satisfied by a wide class of preconditioned Riemannian gradient and projected Sobolev gradient methods in the literature. 
	Moreover, the Lipschitz continuity of the preconditioner required in {\bf (A4)}-(ii) is natural, since the energy functional $E$ is of class $C^2$ and the associated Riemannian gradient is therefore locally Lipschitz continuous. 
	Altogether, assumptions {\bf (A1)}--{\bf (A4)} are standard in numerical simulations and physical experiments, and under {\bf (A1)}--{\bf (A3)} the existence of a minimizer for~\eqref{Riem-Opt-Problem} follows from classical variational arguments (see~\cite{2013Mathematical}).

	\subsection{Properties of the problem}
	Given $\phi\in H_0^1(\mathcal{D})$, we introduce a bounded real-linear operator $\mathcal{H}_{\phi}:H_0^1(\mathcal{D})\to H^{-1}(\mathcal{D})$, for all  $u,v \in H_0^1(\mathcal{D})$
	\begin{align}
		\left\langle\mathcal{H}_{\phi}u,v\right\rangle:=\frac{1}{2}\left(\nabla u,\nabla v\right)_{L^2}+\left(\left(V -\Omega\mathcal{L}_z\right)u,v\right)_{L^2}+\big\langle f(\rho_{\phi})u,v\big\rangle,
	\end{align}
	where 
	$
	\big\langle f(\rho_{\phi})u,v\big\rangle:=\text{Re}\int_{\mathcal{D}}f(\rho_{\phi})u\overline{v}\;\text{d}\bm{x}.
	$
	Then the first and second Fréchet derivatives of the energy functional $E$ can be expressed as
	\begin{align*}
		E'(\phi)=\mathcal{H}_{\phi}\phi\qquad\text{and}\qquad E''(\phi)=\mathcal{H}_{\phi}+f'(\rho_{\phi})\big(|\phi|^2+\phi^2\,\overline{\,\cdot\,}\big).
	\end{align*}
	From a variational perspective, the local minimizer $\phi_g$ satisfies the first-order and second-order necessary conditions:
	\begin{align}
		E'(\phi_g)=\lambda_{\phi_g}\mathcal{I}\phi_g\quad\text{and}\quad 
		\left\langle\big(E''(\phi_g)-\lambda_{\phi_g}\mathcal{I}\big)v,v\right\rangle\ge 0\quad \text{for all}\ v\in T_{\phi_g}\mathcal{M},
	\end{align}
	with $\lambda_{\phi_g}$ being the Lagrange multiplier associated with the $L^2$-normalization constraint and $\mathcal{I}: L^2(\mathcal{D})\to L^2(\mathcal{D})\subset H^{-1}(\mathcal{D})$ the canonical identification $\mathcal{I}v:=(v,\cdot)_{L^2}$. Equivalently,  $\lambda_{\phi_g}=\left\langle\mathcal{H}_{\phi_g}\phi_g,\phi_g\right\rangle$ can be seen as an eigenvalue of the nonlinear eigenproblem $E^{\prime}(\phi)=\lambda \mathcal{I} \phi$ with eigenfunction $\phi_g$.
	In the special case $\Omega=0$ and $f(s)=\eta s,\ \eta\ge0$, and when restricting to real-valued functions, the local minimizer is nondegenerate in the classical sense: the second-order sufficient optimality condition holds, i.e.,
	\begin{align*}
		\left\langle\big(E''(\phi_g)-\lambda_{\phi_g}\mathcal{I}\big)v,v\right\rangle\ge C\|v\|^2_{H^1}\quad \text{for all}\ v\in T_{\phi_g}\mathcal{M}.
	\end{align*}. 
	This condition implies that the local minimizer is isolated. For $\Omega>0$, however, this is no longer true due to symmetry, but we will see that a corresponding coercivity property still holds on a closed subspace of $T_{\phi_g}\mathcal{M}$. 
	Let us next introduce the set of local minimizers at the same energy level as a given local minimizer $\phi_g$: 
	\begin{align}\label{mathcal_S}
		\mathcal{S}:=\Big\{\phi\in \mathcal{M}\;\big|\;\phi\; \text{is a local minimizer and } E(\phi)= E_{\mathcal{S}}:=E(\phi_g)\Big\}.
	\end{align}
	To address symmetry-induced degeneracy, Bott, in his seminal work~\cite{1954Bott}, introduced the notion of nondegenerate critical manifolds, a condition now known as the Morse--Bott condition \cite{2020L-S,2025Fast}. 
	We now recall the precise formulation. 
	\begin{definition}[\textbf{Morse--Bott Condition}]
		Let $ E : \mathcal{M} \to \mathbb{R} $   be a   $C^2$ functional defined on a smooth Riemannian submanifold  $\mathcal{M} \subset X$, where $X$ is a real Hilbert space.
		We say that  $E$  satisfies the local Morse--Bott condition near a local minimizer $\phi_g$ if there exists a sufficiently small $\sigma > 0$ such that the set
		$
		\mathcal{S}_{\sigma}(\phi_g):=\mathcal{S}\cap \mathcal{B}_{\sigma}(\phi_g)
		$
		is a finite-dimensional $C^1$ embedded submanifold of $\mathcal{M}$, and for every $\phi_g' \in \mathcal{S}_{\sigma}(\phi_g)$, we have the identity 
		\begin{align*}
		K_{\phi_g'}\,&:=\,\ker\left( \big(\nabla^{\mathcal{R}}_{X}\big)^2 E(\phi_g') \right)\\
		&:=\,\left\{v\in T_{\phi_g'}\mathcal{M}\;\big|\;\big(\big(\nabla^{\mathcal{R}}_{X}\big)^2 E(\phi_g')v,u\big)_X=0,\;\forall\; u\in T_{\phi_g'}\mathcal{M}\right\} \,=\, T_{\phi_g'} \mathcal{S}_{\sigma}(\phi_g).
		\end{align*}
		Here, $\nabla^{\mathcal{R}}_{X} E(\phi)$ and $\big(\nabla^{\mathcal{R}}_{X}\big)^2 E(\phi)  $ denote respectively the Riemannian gradient and the Riemannian Hessian of $E$ at $\phi$, computed with respect to the Riemannian metric on  $\mathcal{M}$ induced by the ambient Hilbert space $X$.
	\end{definition}
	
	The local Morse--Bott condition stated above is formulated in full generality. In particular, the formulations used in  \cite{2025OnBest,2024Convergence} (also known as quasi-isolated ground states) correspond to specific symmetry-induced scenarios. 
	In those works, the set $\mathcal{S}$ is assumed to be the orbit of $\phi_g$ under symmetry transformations, specifically phase rotations and, in the rotationally invariant setting, spatial rotations. 
	In such cases, the local manifold  $ \mathcal{S}_{\sigma}(\phi_g) $  is precisely the group orbit through $ \phi_g $, which is automatically an embedded submanifold. Moreover, the tangent space at any  $ \phi \in \mathcal{S}_{\sigma}(\phi_g) $ contains the infinitesimal generators of these symmetries; concretely,
	$\operatorname{span}\{\, \ci \phi,\, \ci\mathcal{L}_z \phi \,\} \subseteq T_{\phi} \mathcal{S}_{\sigma}(\phi_g).$
	
	In this work, we go beyond symmetry-induced  manifolds and consider the general setting described by the local Morse--Bott condition above. Crucially, whether this condition holds turns out to be decisive for the convergence behavior and overall efficiency of numerical optimization algorithms. 
	
	We now introduce the  $ \mathcal{P}_{\phi} $-orthogonal complement of  $ K_{\phi} $, denoted by $ R_{\phi} $, i.e.,
	\[
	R_{\phi}:=(K_{\phi})^{\bot}_{\mathcal{P}_{\phi}}=\left\{u\in T_{\phi}\mathcal{M}\;\big|\;(u,v)_{\mathcal{P}_{\phi}}=0,\ \forall\;v\in K_{\phi}\right\}.
	\] 
	Let $ \mathcal{J}_{\phi} : H_0^1(\mathcal{D}) \to R_{\phi} $ denote the  $ \mathcal{P}_{\phi} $-orthogonal projection operator onto  $ R_{\phi} $.
	Under the Morse--Bott condition, for any $\phi_g'\in\mathcal{B}_{\sigma}(\phi_g)$, $E''(\phi_g')-\lambda_{\phi_g'}\mathcal{I}$ is non-degenerate on $ R_{\phi_g'} $. 
	\begin{property}\label{Prop2}
		Let $E$ satisfy the Morse--Bott condition around $\phi_g$. Then, there exists a sufficiently small $\sigma>0$ such that for every $\phi_g' \in \mathcal{S}_{\sigma}(\phi_g)$, the operator $E''(\phi_g')-\lambda_{\phi_g'}\mathcal{I}$ is coercive on $R_{\phi_g'}$, i.e., 
		\begin{align*}
			\big\langle(E''(\phi_g')-\lambda_{\phi_g'}\mathcal{I})v,v\big\rangle\ge C\|v\|^2_{H^1}\quad\ for\ all\ v\in R_{\phi_g'}.
		\end{align*}
	\end{property} 
	\begin{proof}
		The proof is given in Appendix \ref{Appendix_Pro}.
	\end{proof}
	Finally, for any $\phi\in H_0^1(\mathcal{D})$, the important properties of $E(\phi)$ and $E''(\phi)$ are summarized below (cf. \cite[Prop.~2.3]{2025OnBest}).
	\begin{property}\label{Property-of-E''}
		Given $\phi\in H_0^1(\mathcal{D})$ and for all $u, v\in H_0^1(\mathcal{D})$, the following conclusions hold:
		\begin{itemize}[label={}, labelsep=0pt, leftmargin=*]
			\item $(i)$ $E''(\phi)$ is a continuous operator on $H_0^1(\mathcal{D})$, i.e.,
			\begin{align*}
				\left|\big\langle E''(\phi)u,v\big\rangle\right|\le C_{\phi}\|u\|_{H^1}\|v\|_{H^1}.
			\end{align*}
			\item $(ii)$ Given $\psi\in H_0^1(\mathcal{D})$, the following inequality holds
			\begin{align*}
				\left|\left\langle \big(E''(\phi)-E''(\psi)\big)u,v\right\rangle\right|\le C_{\phi,\psi}\|u\|_{H^1}\|v\|_{H^1}\|\phi-\psi\|_{H^1}.
			\end{align*}
			\item $(iii)$ The following Lipschitz-type inequality holds
			\begin{align*}
				E(\phi+v)-E(\phi)\le \big\langle E'(\phi),v\big\rangle+\frac{1}{2}\big\langle E''(\phi)v,v\big\rangle+C_{\phi,v}\|v\|^3_{H^1}.
			\end{align*}
		\end{itemize}
	\end{property}
	
	\subsection{Properties of preconditioned Riemannian gradient iterations} 
    
	For a nondegenerate sequence of step size parameters $\tau_n>0$, 
	a preconditioned Riemannian gradient method takes the form
	\begin{align}\label{Riem-Opt-menthod}
		\phi^{n+1}\,=\,\mathcal{R}_{\phi^n}(\tau_n d_n):=\frac{\phi^n+\tau_n d_n}{\quad\big\|\phi^n+\tau_n d_n\big\|_{L^2}}
	\end{align}
	with the descent direction given by the negative Riemannian gradient in the $\mathcal{P}_{\phi}$-metric:
	\begin{align*}
		d_n:=-\nabla_{\mathcal{P}}^{\mathcal{R}}E(\phi^n)
		=-\mathcal{P}_{\phi^n}^{-1}\mathcal{H}_{\phi^n}\phi^n+\lambda_{\phi^n}\mathcal{P}_{\phi^n}^{-1}\mathcal{I}\phi^n,\quad
		\lambda_{\phi}:=\frac{\big(\phi,\mathcal{P}_{\phi}^{-1}\mathcal{H}_{\phi}\phi\big)_{L^2}}{\big(\phi,\mathcal{P}_{\phi}^{-1}\mathcal{I}\phi\big)_{L^2}}.
	\end{align*}
	Note that the formula exploits that the Riemannian gradient $\nabla_{\mathcal{P}}^{\mathcal{R}}E(\phi)$ is given by 
	\begin{align*}
		\nabla_{\mathcal{P}}^{\mathcal{R}}E(\phi)
		=\text{\normalfont Proj}_{\phi}^{\mathcal{P}_{\phi}} \mathcal{P}^{-1}_{\phi} E^{\prime}(\phi),
		\qquad \mbox{where}
		\quad
		\text{\normalfont Proj}_{\phi}^{\mathcal{P}_{\phi}}(v)
		= v - \frac{ ( \phi , v )_{L^2} }{(\phi, \mathcal{P}_{\phi}^{-1} \mathcal{I} \phi )_{L^2} } \mathcal{P}_{\phi}^{-1} \mathcal{I} \phi.
	\end{align*}
	Although the assumptions on the preconditioner differ from those in \cite{2025OnBest}, the proof of the following basic properties is largely analogous (cf.~\cite[Proposition 3.1]{2025OnBest}); see also Remark~\ref{Diff-Preconditioner}. We briefly recall them here. 
	\begin{property}\label{Property-P}
		Given $\phi\in H_0^1(\mathcal{D})$ and for all $u, v\in H_0^1(\mathcal{D})$ and $w\in H^{-1}(\mathcal{D})$, the following conclusions hold:
		\begin{itemize}[label={}, labelsep=0pt, leftmargin=*]
			\item $(i)$ If $E$ satisfies the Morse--Bott condition around $\phi_g$, then there exists $\sigma>0$ such that for all $\phi \in \mathcal{S}_\sigma(\phi_g)$, the operators $\mathcal{P}_\phi$ and $E''(\phi) - \lambda_\phi \mathcal{I}$ are spectrally equivalent on $R_\phi$, i.e.,
			\begin{align}\label{Mu-L}
				\inf_{v\in R_{\phi} \setminus \{0\}}\frac{\big\langle \big(E''(\phi)-\lambda_{\phi}\mathcal{I}\big)v,v\big\rangle}{\big\langle\mathcal{P}_{\phi}v,v\big\rangle}=\mu>0,\;\,\,\,\,\sup_{v\in R_{\phi} \setminus \{0\}}\frac{\big\langle \big(E''(\phi)-\lambda_{\phi}\mathcal{I}\big)v,v\big\rangle}{\big\langle\mathcal{P}_{\phi}v,v\big\rangle}=L<\infty.
			\end{align}        
			\item $(ii)$ For any $\phi\in\mathcal{M}$, there exists $\sigma>0$ such that for all $\psi\in\mathcal{B}_{\sigma}(\phi)$, the operator $\nabla^{\mathcal{R}}_{\mathcal{P}} E(\cdot):H_0^1(\mathcal{D})\to H_0^1(\mathcal{D})$ and the functional $\lambda_{(\cdot)}:H^1_0(\mathcal{D})\to \mathbb{R}$ are Lipschitz continuous at $\phi$, i.e., 
			\begin{align*}
				\big\|\nabla^{\mathcal{R}}_{\mathcal{P}} E(\phi)-\nabla^{\mathcal{R}}_{\mathcal{P}} E(\psi)\big\|_{H^1}\le C_{\phi}\|\phi-\psi\|_{H^1}\quad \mbox{and}\quad \big|\lambda_{\phi}-\lambda_{\psi}\big|\le C_{\phi}\|\phi-\psi\|_{H^1}.
			\end{align*}
			\item $(iii)$ Let $\phi\in\mathcal{M}$, then for all $v\in T_{\phi}\mathcal{M}$, it holds 
			\begin{align*}
				\big|\mathcal{R}_{\phi}(tv)-(\phi+tv)\big|\le\frac{1}{2}t^2\|v\|^2_{L^2}|\phi+tv|\quad \text{pointwise a.e. in}\ \mathcal{D}.
			\end{align*}
		\end{itemize}
	\end{property}
	\begin{remark}\label{Diff-Preconditioner}
		In contrast to previous work (e.g.,~\cite{2025OnBest,2024Convergence}), we do not assume invariance of the preconditioner under the underlying symmetry group (such as rotations or phase shifts), nor do we impose additional symmetry-based structural assumptions on the critical set. As a consequence, the local constants $L$ and $\mu$ in~\eqref{Mu-L} may depend on the base point $\phi$ along the embedded submanifold. However, our analysis is local in nature. Since the energy functional $E$ is of class $C^2$ and the projection operator $\mathcal{J}_{\phi}$ depends continuously on $\phi \in \mathcal{S}_{\sigma}(\phi_g)$, this variation is well controlled. The continuity of $\mathcal{J}_{\phi}$ follows from the fact that $\mathcal{S}_{\sigma}(\phi_g)$ is a finite-dimensional $C^1$ embedded submanifold, so that the orthogonal projection onto the tangent space $T_{\phi}\mathcal{S}_{\sigma}(\phi_g)$ varies continuously with $\phi$. Combined with the continuity of the preconditioner $\mathcal{P}_{\phi}$, this implies that the projection onto the $\mathcal{P}_{\phi}$-orthogonal complement also depends continuously on $\phi$. Consequently, the constants appearing in our estimates depend only on the size of the neighborhood around the submanifold. Any such dependence can be absorbed into the small parameter $\varepsilon$, which governs both the stepsize restriction and the asymptotic convergence rate. For notational simplicity, we therefore denote all such local bounds uniformly by $L$ and $\mu$.
	\end{remark}
	
	\section{Sharp results on convergence and structure}\label{Sec3} 
	In this section, we establish a set of sharp theoretical characterizations describing the local behavior and geometric structure of minimizers for the Gross--Pitaevskii energy functional.  First, in the case of non-isolated minimizers, we show that the Morse--Bott condition provides a sufficient condition for the set of ground states to be partitioned into finitely many embedded submanifolds, on each of which the energy functional is constant. 
	Second, for the P-RG iterations, we derive its optimal local convergence rate and prove that local $Q$-linear convergence of the energy sequence occurs if and only if the Morse--Bott condition holds, thereby characterizing precisely when such fast rates are attainable. 
	Furthermore, when restricted to the ground state set, the Morse--Bott condition holds if and only if this set decomposes into finitely many symmetry orbits generated by the phase and rotational invariances. In this case, the P-RG iterations exhibit local linear convergence in a neighborhood of the ground state set, thereby connecting the geometric regularity of the critical manifold, the topological classification of minimizers, and the local convergence behavior of the iteration.  
	Finally, when $E$ is real analytic and the Morse--Bott condition fails, the P-RG iterates converge locally at a sublinear rate.
	\subsection{Main results}
	
	\subsubsection{Morse--Bott condition for the finite classification of global minimizers}
	In the presence of continuous symmetries, such as phase invariance and spatial rotations, the Gross--Pitaevskii energy functional typically admits non-unique global minimizers that are related by symmetry transformations. These states organize into continuous families, or orbits, each forming a compact $C^1$ embedded submanifold. A fundamental question arises: under what conditions can all physically distinct global minimizers be cleanly separated into such orbits, so that no two orbits intersect or accumulate arbitrarily close to one another? 
	
	Such a classification is not only essential for a rigorous understanding of the solution landscape, but also crucial for the design of optimization algorithms that aim to recover all relevant physical configurations. In what follows, we introduce a notion of well-defined classification, and show that the Morse--Bott condition provides a sufficient geometric criterion for this property.

	\begin{definition}[\textbf{Well-defined classification}]
		We say that the set of global minimizers $\mathcal{S}_g$ admits a well-defined classification if it can be written as a finite disjoint union of distinct connected symmetry orbits:
		\begin{align*}
			\mathcal{S}_g
			= \bigsqcup_{\ell=1}^N \mathcal{S}_{\phi_{g,\ell}}
			= \bigsqcup_{\ell=1}^N 
			\left\{ \psi \in \mathcal{M} \,\big|\, 
			\psi = e^{\ci \alpha} \phi_{g,\ell}(A_{\beta} \bm{x}),\ 
			\alpha,\beta\in[-\pi,\pi)
			\right\},
		\end{align*}
		for some finite collection of ground state representatives $\phi_{g,\ell}\in\mathcal{S}_g$, with $1 \le \ell \le N$.
	\end{definition}
	This notion of classification captures the idea that global minimizers can be grouped into finitely many geometrically distinct families, each closed under the inherent symmetries of the system. From a computational perspective, it guarantees that the solution landscape is enumerable and resolvable: with a finite number of appropriately chosen initial guesses, one can in principle recover every inequivalent global minimizer using standard optimization algorithms. The following theorem provides a geometric criterion for when such a classification holds. 
	
	\begin{thm}\label{finite-classified-thm}
		Suppose the energy functional $E$ satisfies the Morse--Bott condition in a neighborhood of each symmetry orbit
		\[
		\mathcal{S}_{\phi_g}:=\left\{\psi\in\mathcal{M}\mid
		\psi=e^{\ci\alpha}\phi_g(A_{\beta}(\bm{x})),\;
		\alpha,\beta\in[-\pi,\pi)
		\right\},\quad \phi_g\in\mathcal{S}_g.
		\]
		Then the set of global minimizers $\mathcal{S}_g$ admits a well-defined classification. 
	\end{thm}
	In other words, the Morse--Bott condition prevents the appearance of infinitely many distinct families of minimizers or accumulation of symmetry orbits, and enforces a finite decomposition of the ground state set into well-separated components. The proof of \textbf{Theorem \ref{finite-classified-thm}} is given in Section \ref{proofs-main-results}.
	
	\subsubsection{Morse--Bott condition and optimal local linear convergence}
	The following theorem establishes sharp local convergence rates for  preconditioned Riemannian gradient (P-RG) iterations. The proof is given in Section \ref{proofs-main-results}.
	
	\begin{thm}\label{Opt-convergence}
		Let $E$ satisfy the Morse--Bott condition around a local minimizer $\phi_g\in\mathcal{S}$. Then, for every sufficiently small $\varepsilon > 0$, there exists $\sigma>0$ such that for all $\phi^0 \in \mathcal{B}_{\sigma}(\phi_g)$, the sequence $\{\phi^n\}_{n\in\mathbb{N}}$ generated by the P-RG converges $Q$-linearly to a local minimizer $\phi_g^*$ (depending on $\phi^0$) and satisfies $\|\phi^*_g-\phi_g\|_{H^1}\le C\sigma$, i.e.,
		\begin{align*}
			\|\phi^n-\phi^*_g\|_{\mathcal{P}_{\phi^*_g}} \,\le\, \|\phi^{n-1}-\phi^*_g\|_{\mathcal{P}_{\phi^*_g}}\left(\max\left\{|1-\tau\mu|,|1-\tau L |\right\}+\varepsilon\right),
		\end{align*}
		for all $\tau\in(0,2/(L+\varepsilon)),\; n\ge1.$
		In particular, with the asymptotically optimal choice $\tau = 2/(L + \mu)$, the P-RG yields the optimal Q-linear convergence rate
		{\rm\begin{align}
				\|\phi^n-\phi^*_g\|_{\mathcal{P}_{\phi^*_g}} \,\le\, \|\phi^{n-1}-\phi^*_g\|_{\mathcal{P}_{\phi^*_g}}\left(\frac{L-\mu}{L+\mu}+\varepsilon\right),\quad\forall\; n\ge1.
		\end{align}}
	\end{thm}
	This result provides a natural generalization of the classical sharp linear convergence theory for gradient descent on strongly convex quadratic problems to a nonconvex setting governed by the Morse--Bott condition. In particular, the contraction factor $(L-\mu)/(L+\mu)$ coincides with the best possible rate known for gradient descent on strongly convex quadratics, which was shown to be optimal in \cite[Theorem~3 in Chapter~1, Section~4]{1987Polyak}. 
	
	The theorem therefore shows that, despite the presence of continuous symmetries and the resulting degeneracy of the Hessian along symmetry directions, the P-RG iterations with general preconditioners retain the same optimal local $Q$-linear convergence behavior as in the classical strongly convex setting where second-order sufficient conditions hold.
	
	The following theorem, again proved in Section \ref{proofs-main-results}, provides a sharp characterization of the Morse--Bott condition in terms of $Q$-linear convergence of the energy for the P-RG iterations.
	\begin{thm}\label{MB-QE}
		The energy functional $E$ satisfies the Morse--Bott condition around a local minimizer $\phi_g\in\mathcal{S}$ \emph{if and only if} for every sufficiently small $\varepsilon>0$, there exist constants $ \rho \in (0,1) $ and $ \sigma > 0 $ such that for all  $ \phi^0 \in \mathcal{B}_\sigma(\phi_g) $, the energy sequence $\{E(\phi^n)\}_{n\in\mathbb{N}}$ generated by the P-RG converges Q-linearly to $E(\phi_g)$, i.e.,
		\[
		E(\phi^{n+1})-E(\phi_g) \,\le\, (\rho+\varepsilon)\big(E(\phi^{n})-E(\phi_g)\big) , \quad \forall\, \tau\in(0,2/(L+\varepsilon)),\; n \ge 0.
		\]
	\end{thm}
	This naturally raises the question of whether the symmetry-generated critical manifold itself satisfies the Morse--Bott condition. This question is motivated by extensive numerical evidence in important models such as the BEC model, which consistently indicate that the non-uniqueness of ground states arises from the action of the symmetry group \( U(1) \times SO(2) \), i.e., global phase shifts and spatial rotations, with no further degeneracies observed. 
	
	The following result provides a theoretical explanation for this phenomenon. In particular, it shows that if the Morse--Bott condition holds along the symmetry-generated critical manifolds, then every continuous family of ground states is induced by these symmetries. Under assumptions \textbf{(A1)}--\textbf{(A4)}, this rules out additional bifurcation-type or accidental degeneracies. The following theorem makes this connection precise.

	\begin{thm}\label{MB-CL}
		The energy functional $ E $  satisfies the Morse--Bott condition along every symmetry orbit 
		\[
		\mathcal{S}_{\phi_g} \,:=\,\left\{\psi\in\mathcal{M}\mid
		\psi=e^{\ci\alpha}\phi_g(A_\beta(\bm{x})),\;
		\alpha,\beta\in[-\pi,\pi)
		\right\}, 
		\quad \phi_g\in\mathcal{S}_g.
		\]  
		\emph{if and only if}
		\begin{itemize}
			\item [] $(i)$ The set of global minimizers $\mathcal{S}_g$ admits a well-defined classification.
			\item [] $(ii)$ For every sufficiently small $\varepsilon > 0$, there exist constants  $ \rho \in (0,1) $ and $ \sigma > 0 $ such that for all  $ \phi^0 \in \mathcal{B}_\sigma(\mathcal{S}_g) $, the sequence  $ \{\phi^n\}_{n\in\mathbb{N}}$ generated by the P-RG converges linearly to a ground state $\phi_g^*$, which depends on $\phi^0$ and satisfies $\|\phi^*_g-\phi_g\|_{H^1}\le C\sigma$, i.e.,  
			\[
			\|\phi^n - \phi^*_g\|_{H^1} \,\le \,C_{\varepsilon} \, \|\phi^0 - \phi^*_g\|_{H^1} \, (\rho + \varepsilon)^n, \quad \forall\, \tau\in(0,2/(L+\varepsilon)),\; n \ge 0.
			\] 
		\end{itemize}
		
	\end{thm}
	Again, the proof is postponed to Section \ref{proofs-main-results}.

	\subsubsection{Beyond the Morse--Bott condition}
	The previous results rely crucially on the Morse--Bott condition, which ensures that the critical set of the energy functional $E$ is a submanifold and that the Hessian of $E$ is non-degenerate in the normal directions. When the Morse--Bott condition fails, the local geometry of the energy functional \(E\) near local minimizers may exhibit higher-order degeneracies, leading to more intricate dynamical behavior of the P-RG. In particular, by {\bf Theorem~\ref{MB-CL}}, local linear convergence is no longer possible in this setting. Nevertheless, if the nonlinearity \(f\) in \(E\) is real-analytic, which is indeed the case for the Bose--Einstein condensation model since \(f(s) = \eta s\) is linear and hence analytic, then the P-RG iteration still enjoys a weaker but meaningful convergence guarantee: global convergence to a critical point with a sublinear asymptotic rate. This is made precise in the following result, to be proved in Section \ref{proofs-main-results}.
	
	\begin{thm}\label{sublinear-convergence}
		Suppose that the Morse--Bott condition fails for $E$, and that the nonlinearity \(f\) is real-analytic. Then, for every sufficiently small $\varepsilon>0$, there exists $\sigma > 0$ such that for any step size $\tau\in(0,2/(L+\varepsilon))$ and any initial iterate $\phi^0 \in \mathcal{B}_{\sigma}(\text{Crit}(E))$, the sequence $\{\phi^n\}_{n \in \mathbb{N}}$ generated by the P-RG converges to a critical point $\phi_s \in \text{Crit}(E)$. Moreover, the convergence rate is sublinear:
		\[
		\|\phi^n-\phi_s\|_{H^1} \le Cn^{-\frac{\nu}{1-2\nu}}, \quad \forall\; n \ge 0,
		\]
		where $\nu\in(0,1/2)$ is the Łojasiewicz exponent associated with $E$ at the limiting critical point $\phi_s$, independent of $\mathcal{P}_{\phi}$.
	\end{thm}
	Consequently, an important implication is that preconditioning cannot improve the asymptotic convergence order of the P-RG, since the Łojasiewicz exponent (and thus the sublinear rate) does not depend on the preconditioner $\mathcal{P}_{\phi}$. Moreover, this result implies that any P-RG method that is known to converge globally must in fact converge strongly along the entire sequence, rather than only along subsequences. In particular, this applies to the Bose--Einstein condensation model.
	
	\subsection{Technical lemmas}
	Before presenting the proofs of the main results, we introduce several key lemmas that will be instrumental in establishing various aspects of our results. 
	\begin{lemma}\label{Equal-L}
		Given a local minimizer $\phi_g$ of $E$, the following equality holds for any preconditioner $\mathcal{P}_{\phi_g}$:
		\[
		\sup_{v \in R_{\phi_g} \setminus \{0\}}
		\frac{\langle (E''(\phi_g) - \lambda_{\phi_g} \mathcal{I}) v, v \rangle}{\langle \mathcal{P}_{\phi_g} v, v \rangle}=\sup_{v \in T_{\phi_g}\mathcal{M} \setminus \{0\}}
		\frac{\langle (E''(\phi_g) - \lambda_{\phi_g} \mathcal{I}) v, v \rangle}{\langle \mathcal{P}_{\phi_g} v, v \rangle}.
		\]
	\end{lemma}
	\begin{proof}
		For any $v \in T_{\phi_g}\mathcal{M}$, we have the decomposition $v = v_1 + v_2$ with $v_1 \in R_{\phi_g}$ and $v_2 \in K_{\phi_g}$. Noting that, by the self-adjointness of $E''(\phi_g) - \lambda_{\phi_g} \mathcal{I}$ with respect to the duality pairing,
		\[
		\frac{\langle (E''(\phi_g) - \lambda_{\phi_g} \mathcal{I}) v, v \rangle}{\langle \mathcal{P}_{\phi_g} v, v \rangle}
		= 
		\frac{\langle (E''(\phi_g) - \lambda_{\phi_g} \mathcal{I}) v_1, v_1 \rangle}{\langle \mathcal{P}_{\phi_g} v_1, v_1 \rangle + \langle \mathcal{P}_{\phi_g} v_2, v_2 \rangle},
		\]
		it follows that
		\[
		\sup_{v \in T_{\phi_g}\mathcal{M} \setminus \{0\}}
		\frac{\langle (E''(\phi_g) - \lambda_{\phi_g} \mathcal{I}) v, v \rangle}{\langle \mathcal{P}_{\phi_g} v, v \rangle}
		\le
		\sup_{v \in R_{\phi_{g}} \setminus \{0\}}
		\frac{\langle (E''(\phi_g) - \lambda_{\phi_g} \mathcal{I}) v, v \rangle}{\langle \mathcal{P}_{\phi_g} v, v \rangle}.
		\]
		Consequently, we obtain
		\[
		\sup_{v \in R_{\phi_{g}} \setminus \{0\}}
		\frac{\langle (E''(\phi_g) - \lambda_{\phi_g} \mathcal{I}) v, v \rangle}{\langle \mathcal{P}_{\phi_g} v, v \rangle}
		=
		\sup_{v \in T_{\phi_g}\mathcal{M} \setminus \{0\}}
		\frac{\langle (E''(\phi_g) - \lambda_{\phi_g} \mathcal{I}) v, v \rangle}{\langle \mathcal{P}_{\phi_g} v, v \rangle}.
		\] 
	\end{proof}
	\begin{lemma}\label{P-L-In}
		Let $E$ satisfy the Morse--Bott condition around $\phi_g\in\mathcal{S}$. Then, for every sufficiently small $\varepsilon>0$, there exists $\sigma>0$ such that for any $\phi\in\mathcal{B}_{\sigma}(\phi_g)$, the Polyak--\L ojasiewicz inequality holds
		\begin{align*}
			E(\phi)-E(\phi_g)\le \frac{1}{2(\mu-\varepsilon)}\left\|\nabla^{\mathcal{R}}_{\mathcal{P}}E(\phi)\right\|^2_{\mathcal{P}_{\phi}}.
		\end{align*}
		
	\end{lemma}
	\begin{proof}
		For some sufficiently small  $ \sigma > 0 $ , consider the projection of  $ \phi \in \mathcal{B}_{\sigma}(\phi_g) $ onto $ \mathcal{S} $, defined by
		\[
		\phi_g' := \argmin_{u \in \mathcal{S}} \frac{1}{2} \|\phi - u\|_{\mathcal{P}_{\phi_g}}^2.
		\]
		We first establish the existence of  $ \phi_g' $  for sufficiently small  $ \sigma > 0 $. 
		To this end, introduce the $ \mathcal{P}_{\phi_g} $-metric neighborhood
		\[
		\mathcal{B}_{\sigma}^{\mathcal{P}}(\phi_g) := \left\{ \psi \in \mathcal{M} \,\big|\, \|\psi - \phi_g\|_{\mathcal{P}_{\phi_g}} < \sigma \right\}\quad\text{and}\quad\mathcal{S}^{\mathcal{P}}_{\sigma}(\phi_g) := \mathcal{S} \cap \mathcal{B}_{\sigma}^{\mathcal{P}}(\phi_g).
		\]
		Since the energy functional  $ E $  satisfies the Morse--Bott condition around $\phi_g$, there exists $\sigma_0 > 0$  such that
		$ \mathcal{S}^{\mathcal{P}}_{\sigma_0}(\phi_g) $  is a smooth embedded submanifold of  $ \mathcal{M}$, and
		$ E(\phi) \geq E(\phi_g) $  for all  $ \phi \in \mathcal{B}^{\mathcal{P}}_{\sigma_0}(\phi_g)$. Consequently, for any  $ \phi \in \mathcal{B}^{\mathcal{P}}_{\sigma_0}(\phi_g) $, we have  $ E(\phi) = E(\phi_g) $ if and only if $ \phi \in \mathcal{S} $. Now define the local set
		\[
		\overline{\mathcal{S}}^{\mathcal{P}}_{\sigma_0/2}(\phi_g) := \mathcal{S} \cap \overline{\mathcal{B}}^{\mathcal{P}}_{\sigma_0/2}(\phi_g)=E^{-1}(\{ E(\phi_g)\})\cap \overline{\mathcal{B}}^{\mathcal{P}}_{\sigma_0/2}(\phi_g),
		\]    
		this set is closed as the intersection of two closed sets.
		Because  $ \mathcal{S}^{\mathcal{P}}_{\sigma_0/2}(\phi_g) $ is a finite-dimensional embedded $C^1$ submanifold,  
		$\overline{\mathcal{S}}^{\mathcal{P}}_{\sigma_0/2}(\phi_g) $ is compact. For any  $ \phi \in \mathcal{B}^{\mathcal{P}}_{\sigma_0/4}(\phi_g)  $  and any  $ u \in \mathcal{S} \setminus \overline{\mathcal{S}}^{\mathcal{P}}_{\sigma_0/2}(\phi_g)  $ , we estimate
		\[
		\|\phi - u\|_{\mathcal{P}_{\phi_g}} 
		\geq \|u - \phi_g\|_{\mathcal{P}_{\phi_g}} - \|\phi - \phi_g\|_{\mathcal{P}_{\phi_g}} 
		> \frac{\sigma_0}{2} - \frac{\sigma_0}{4} = \frac{\sigma_0}{4} 
		> \|\phi - \phi_g\|_{\mathcal{P}_{\phi_g}}.
		\]
		Hence, the local minimizer of  $ \|\phi - u\|_{\mathcal{P}_{\phi_g}} $  over  $ u \in \mathcal{S} $  must lie in  $\overline{\mathcal{S}}^{\mathcal{P}}_{\sigma_0/2}(\phi_g) $. 
		Since the norm is continuous and  $\overline{\mathcal{S}}^{\mathcal{P}}_{\sigma_0/2}(\phi_g) $ is compact, the minimum is attained. Therefore, by \textbf{(A4)}-$(i)$, there exists sufficiently small $\sigma>0$ such that for all  $ \phi \in \mathcal{B}_{\sigma}(\phi_g) $ , the projection $ \phi_g' $ exists.
		Moreover, since  $\phi_g' \in \mathcal{S}^{\mathcal{P}}_{\sigma_0}(\phi_g) $,
		it satisfies the first-order optimality condition:
		\[
		(\phi - \phi_g', v)_{\mathcal{P}_{\phi_g}} = 0, \quad \forall\, v \in T_{\phi_g'} \mathcal{S}^{\mathcal{P}}_{\sigma_0}(\phi_g).
		\]
		By assumption \textbf{(A4)}-$(ii)$ and the following inequality
		\[
		\|\phi_g' - \phi_g\|_{\mathcal{P}_{\phi_g}} 
		\leq \|\phi - \phi_g\|_{\mathcal{P}_{\phi_g}} + \|\phi - \phi_g'\|_{\mathcal{P}_{\phi_g}} 
		\leq 2 \|\phi - \phi_g\|_{\mathcal{P}_{\phi_g}},
		\]
		we deduce that as  $\sigma \to 0^+$,
		$
		(\phi - \phi_g', v)_{\mathcal{P}_{\phi_g'}} = o(\|\phi - \phi_g'\|_{H^1}), \forall\; v \in T_{\phi_g'} \mathcal{S}^{\mathcal{P}}_{\sigma_0}(\phi_g).
		$
		Combining the following decomposition
		\[
		\phi - \phi_g' = \operatorname{Proj}^{L^2}_{\phi_g'}(\phi - \phi_g') - \frac{1}{2} \|\phi - \phi_g'\|_{L^2}^2 \, \phi_g',
		\]
		we have
		\begin{align}\label{kernel-vanish}
			\phi - \phi_g' = \mathcal{J}_{\phi_g'}(\phi - \phi_g') + o(\|\phi - \phi_g'\|_{H^1}) \quad \text{as } \sigma \to 0^+,
		\end{align}
		where we recall $\mathcal{J}_{\phi_g'}$ as the $\mathcal{P}_{\phi_g'}$-orthogonal projection onto the normal space $R_{\phi_g'}$.
		The remainder of the proof follows verbatim from \cite[Lemma~4.2]{2025OnBest}. Specifically, according to $E(\phi_g')=E(\phi_g)$, Taylor's formula at $\phi$, and \eqref{kernel-vanish}, we have
		
		\begin{align*}
			E&(\phi)-E(\phi_g)=\left\langle E'(\phi),\phi-\phi_g'\right\rangle-\frac{1}{2}\left\langle E''(\phi)(\phi-\phi_g'),\phi-\phi_g'\right\rangle+o\big(\|\phi-\phi_g'\|^2_{H^1}\big)\\
			=&\big( \nabla^{\mathcal{R}}_{\mathcal{P}} E(\phi),\phi-\phi_g'\big)_{\mathcal{P}_{\phi}}\hspace{-0.32cm}-\frac{1}{2}\left\langle \big(E''(\phi)-\lambda_{\phi}\mathcal{I}\big)(\phi-\phi_g'),\phi-\phi_g'\right\rangle+o\big(\|\phi-\phi_g'\|^2_{H^1}\big)\\
			=&\big( \nabla^{\mathcal{R}}_{\mathcal{P}} E(\phi), \mathcal{J}_{\phi_g'}(\phi - \phi_g')\big)_{\mathcal{P}_{\phi}}\hspace{-0.32cm}-\frac{1}{2}\big\langle \big(E''(\phi)-\lambda_{\phi}\mathcal{I}\big) \mathcal{J}_{\phi_g'}(\phi - \phi_g'), \mathcal{J}_{\phi_g'}(\phi - \phi_g')\big\rangle+o\big(\|\phi-\phi_g'\|^2_{H^1}\big).
		\end{align*}
		
		Based on {\bf Property \ref{Property-of-E''}}-$(ii)$, {\bf Property \ref{Property-P}}-$(ii)$, and {\bf (A4)}-$(ii)$, the following estimates hold
		\begin{align*}
			\left\langle \big(E''(\phi)-E''(\phi_g')\big)\mathcal{J}_{\phi_g'}(\phi-\phi_g'),\mathcal{J}_{\phi_g'}(\phi-\phi_g')\right\rangle&=o\big(\|\phi-\phi_g'\|^2_{H^1}\big),\\
			\left\langle \big(\lambda_{\phi_g'}\mathcal{I}-\lambda_{\phi}\mathcal{I}\big)\mathcal{J}_{\phi_g'}(\phi-\phi_g'),\mathcal{J}_{\phi_g'}(\phi-\phi_g')\right\rangle&=o\big(\|\phi- \phi_g' \|^2_{H^1}\big), \\ 
			\left\langle\big(\mathcal{P}_{\phi}-\mathcal{P}_{\phi_g'}\big)\mathcal{J}_{\phi_g'}(\phi-\phi_g'),\mathcal{J}_{\phi_g'}(\phi-\phi_g')\right\rangle&=o\big(\|\phi-\phi_g'\|^2_{H^1}\big).
		\end{align*}
		According to {\bf Property \ref{Property-P}}-$(i)$, the following lower bound estimate holds
		\begin{align*}
			\frac{\left\langle\big( E''(\phi_g')-\lambda_{\phi_g'}\mathcal{I}\big)\mathcal{J}_{\phi_g'}(\phi-\phi_g'),\mathcal{J}_{\phi_g'}(\phi-\phi_g')\right\rangle}{\left\langle\mathcal{P}_{\phi_g'}\mathcal{J}_{\phi_g'}(\phi-\phi_g'),\mathcal{J}_{\phi_g'}(\phi-\phi_g')\right\rangle}\ge\mu.
		\end{align*}
		In summary, the estimate we want is derived
		\begin{align*}
			-\frac{1}{2}&\left\langle \big(E''(\phi)-\lambda_{\phi}\mathcal{I}\big)\mathcal{J}_{\phi_g'}(\phi-\phi_g'),\mathcal{J}_{\phi_g'}(\phi-\phi_g')\right\rangle\\
			&\qquad\qquad\qquad\le-\frac{\mu}{2}\left\langle\mathcal{P}_{\phi}\mathcal{J}_{\phi_g'}(\phi-\phi_g'),\mathcal{J}_{\phi_g'}(\phi-\phi_g')\right\rangle
			+o\big(\|\phi-\phi_g'\|^2_{H^1}\big).
		\end{align*}
		Then, combined with \eqref{kernel-vanish}, we further get
		\begin{align*}
			E(\phi)-E(\phi_g)\le\big(\nabla^{\mathcal{R}}_{\mathcal{P}} E(\phi),\mathcal{J}_{\phi_g'}(\phi-\phi_g')\big)_{\mathcal{P}_{\phi}}\hspace{-0.32cm}-\frac{\mu}{2}\big\|\mathcal{J}_{\phi_g'}(\phi-\phi_g')\big\|_{\mathcal{P}_{\phi}}+o\big(\|\mathcal{J}_{\phi_g'}(\phi-\phi_g')\|^2_{H^1}\big).
		\end{align*}
		By the minimization property of $\phi_g'$, we obtain
		\[
		\|\phi-\phi_g'\|_{H^1}\le C\|\phi-\phi_g\|_{H^1}.
		\]
		Consequently, for all sufficiently small $\varepsilon>0$, there exists $\sigma>0$ such that for any $\phi \in \mathcal{B}_{\sigma}(\phi_g)$, the Polyak--\L ojasiewicz inequality is deduced as follows
		\begin{align*}
			E(\phi)-E(\phi_g)&\le\left(\nabla^{\mathcal{R}}_{\mathcal{P}} E(\phi),\mathcal{J}_{\phi_g'}(\phi-\phi_g')\right)_{\mathcal{P}_{\phi}}-\frac{\mu-\varepsilon}{2}\left(\mathcal{J}_{\phi_g'}(\phi-\phi_g'),\mathcal{J}_{\phi_g'}(\phi-\phi_g')\right)_{\mathcal{P}_{\phi}}\\
			&\le\sup\limits_{v\in H_0^1(\mathcal{D})}\Bigg(\left( \nabla^{\mathcal{R}}_{\mathcal{P}} E(\phi),v\right)_{\mathcal{P}_{\phi}}-\frac{\mu-\varepsilon}{2}(v,v)_{\mathcal{P}_{\phi}}\Bigg)\\
			&=\frac{1}{2(\mu-\varepsilon)}\left\|\nabla^{\mathcal{R}}_{\mathcal{P}} E(\phi)\right\|^2_{\mathcal{P}_{\phi}}.
		\end{align*}
	\end{proof}
	\begin{lemma}\label{PL-Co}
		Let $E$ satisfy the Polyak--\L ojasiewicz inequality around $\phi_g\in\mathcal{S}$, i.e., there exists $\sigma>0$ and a constant $\mu_{PL}>0$ such that for any $\phi\in\mathcal{B}_{\sigma}(\phi_g)$, the following inequality holds
		\begin{align*}
			E(\phi)-E(\phi_g)\le \frac{1}{2\mu_{PL}}\left\|\nabla^{\mathcal{R}}_{\mathcal{P}}E(\phi)\right\|^2_{\mathcal{P}_{\phi}}.
		\end{align*}
		Then, for every  $ \phi_g' \in \mathcal{S}_{\sigma}(\phi_g) $ , the Riemannian Hessian  $ (\nabla_{\mathcal{P}}^{\mathcal{R}})^2 E(\phi_g') $  is uniformly coercive on $ R_{\phi_g'} $ , i.e.,
		\begin{align*}
			\left( (\nabla_{\mathcal{P}}^{\mathcal{R}})^2 E(\phi_g') v, v \right)_{\mathcal{P}_{\phi_g'}} \geq \mu_{\mathrm{PL}} \|v\|_{\mathcal{P}_{\phi_g'}}^2 \quad \text{for all } v \in R_{\phi_g'}.
		\end{align*}
	\end{lemma}
	
	\begin{proof}
		The Riemannian Hessian at  $ \phi_g'\in\mathcal{B}_{\sigma}(\phi_g) $  is given by
		\[
		(\nabla_{\mathcal{P}}^{\mathcal{R}})^2 E(\phi_g') 
		= \mathrm{Proj}_{\phi_g'}^{\mathcal{P}_{\phi_g'}} \, \mathcal{P}_{\phi_g'}^{-1} \big( E''(\phi_g') - \lambda_{\phi_g'} \mathcal{I} \big) \big|_{T_{\phi_g'} \mathcal{M}}.
		\]
		Observe that the kernel $ K_{\phi_g'} $  is independent of the choice of the Riemannian metric. Indeed, for any  $ v, u \in T_{\phi_g'} \mathcal{M} $,
		\[
		\big( (\nabla_{\mathcal{P}}^{\mathcal{R}})^2 E(\phi_g') v, u \big)_{\mathcal{P}_{\phi_g'}} 
		= \big\langle (E''(\phi_g') - \lambda_{\phi_g'} \mathcal{I}) v, u \big\rangle.
		\]
		Hence, to study the kernel $K_{\phi_g'}$, we may choose the convenient  $ H_0^1 $-metric, i.e.,  $ \mathcal{P}_{\phi_g'} = -\Delta $. Under this choice, the Riemannian Hessian becomes
		\begin{align*}
			&(\nabla_{H_0^1}^{\mathcal{R}})^2 E(\phi_g')
			= \mathrm{Proj}_{\phi_g'}^{H_0^1} \, (-\Delta)^{-1} \big( E''(\phi_g') - \lambda_{\phi_g'} \mathcal{I} \big) \big|_{T_{\phi_g'} \mathcal{M}} \\
			&= \mathrm{Proj}_{\phi_g'}^{H_0^1} \left( \tfrac{1}{2} I 
			+ (-\Delta)^{-1} \Bigl( V - \Omega \mathcal{L}_z + f(\rho_{\phi_g'}) 
			+ f'(\rho_{\phi_g'}) \bigl( |\phi_g'|^2 + (\phi_g')^2 \overline{\cdot} \bigr) 
			- \lambda_{\phi_g'} \mathcal{I} \Bigr) \right) \big|_{T_{\phi_g'} \mathcal{M}} \\
			&=: \tfrac{1}{2} I \big|_{T_{\phi_g'} \mathcal{M}} + \mathcal{A}_{\phi_g'} \big|_{T_{\phi_g'} \mathcal{M}}.
		\end{align*}
		The operator $\mathcal{A}_{\phi_g'}$ is compact due to 
		\begin{align*}
			\|\mathcal{A}_{\phi_g'}v\|^2_{H_0^1}&=\left\langle\Bigl( V - \Omega \mathcal{L}_z + f(\rho_\phi) + f'(\rho_\phi)\bigl(|\phi|^2 + \phi^2 \overline{\cdot}\,\bigr)-\lambda_{\phi_g'}\mathcal{I} \Bigr)v,\mathcal{A}_{\phi_g'}v\right\rangle\\
			&\le C\left(\|v\|_{L^2}+\|v\|_{L^{6/(4-\theta)}}\right)\|\mathcal{A}_{\phi_g'}v\|_{H^1_0},\quad \forall\; v\in T_{\phi_g'}\mathcal{M}.
		\end{align*}
		Therefore, the $H_0^1$-Riemannian Hessian is a compact perturbation of the identity operator and hence Fredholm of index zero. Its kernel $ K_{\phi_g'}$ is finite-dimensional.
		
		As in the proof of \cite[Proposition~2.2]{2025OnBest}, we now establish pointwise coercivity on  $ R_{\phi_g'}$
		\[
		\mu_{\phi_g'} := \inf_{v \in R_{\phi_g'} \setminus \{0\}} 
		\frac{\big( (\nabla_{\mathcal{P}}^{\mathcal{R}})^2 E(\phi_g') v, v \big)_{\mathcal{P}_{\phi_g'}}}
		{\|v\|_{\mathcal{P}_{\phi_g'}}^2}
		= \inf_{v \in R_{\phi_g'} \setminus \{0\}} 
		\frac{\big\langle (E''(\phi_g') - \lambda_{\phi_g'} \mathcal{I}) v, v \big\rangle}
		{\|v\|_{\mathcal{P}_{\phi_g'}}^2} > 0.
		\]
		For any  $ v \in R_{\phi_g'} $  with  $ \|v\|_{\mathcal{P}_{\phi_g'}} = 1 $  and sufficiently small  $ t \in \mathbb{R} $  such that  $ \phi = \mathcal{R}_{\phi_g'}(t v) \in \mathcal{B}_{\sigma}(\phi_g) $ , the Taylor expansion yields
		\[
		E(\phi) - E(\phi_g') 
		= \frac{t^2}{2} \big( (\nabla_{\mathcal{P}}^{\mathcal{R}})^2 E(\phi_g') v, v \big)_{\mathcal{P}_{\phi_g'}} + o(t^2).
		\]
		Since the Polyak--\L ojasiewicz inequality holds by assumption, we have
		\[
		E(\phi) - E(\phi_g') 
		= E(\phi) - E(\phi_g) 
		\le \frac{1}{2\mu_{PL}} \big\| \nabla^{\mathcal{R}}_{\mathcal{P}} E(\phi) \big\|_{\mathcal{P}_{\phi}}^2.
		\]
		Noting that  $ \nabla^{\mathcal{R}}_{\mathcal{P}} E(\phi) = t (\nabla_{\mathcal{P}}^{\mathcal{R}})^2 E(\phi_g') v + o(t) $, we obtain
		\[
		\frac{t^2}{2} \big( (\nabla_{\mathcal{P}}^{\mathcal{R}})^2 E(\phi_g') v, v \big)_{\mathcal{P}_{\phi_g'}} + o(t^2)
		\le \frac{t^2}{2\mu_{PL}} \left\| (\nabla_{\mathcal{P}}^{\mathcal{R}})^2 E(\phi_g') v \right\|_{\mathcal{P}_{\phi_g'}}^2 + o(t^2).
		\]
		Dividing both sides by  $ t^2/2 $, letting  $ t \to 0 $, and noting that  $ (\nabla_{\mathcal{P}}^{\mathcal{R}})^2 E(\phi_g') $  is a bounded self-adjoint operator, we obtain
		\begin{align*}
			\mu_{PL}&\le\frac{1}{\mu_{\phi_g'}}\inf_{v\in R_{\phi_g'}\backslash\{0\}}\frac{\big( (\nabla_{\mathcal{P}}^{\mathcal{R}})^2 E(\phi_g')v,(\nabla_{\mathcal{P}}^{\mathcal{R}})^2 E(\phi_g')v\big)_{\mathcal{P}_{\phi_g'}}}{(v,v)_{\mathcal{P}_{\phi_g'}}}\\
			&=\frac{1}{\mu_{\phi_g'}}\inf\;\operatorname{Spec}\left( \big((\nabla_{\mathcal{P}}^{\mathcal{R}})^2 E(\phi_g')\big)^2 \big|_{R_{\phi_g'}} \right)=\mu_{\phi_g'}.
		\end{align*}
		This completes the proof.
	\end{proof}
	\begin{lemma}\label{PL-MB}
		Let $E$ satisfy the Polyak--\L ojasiewicz inequality around $\phi_g\in\mathcal{S}$, i.e., there exists $\sigma>0$ and a constant $\mu_{PL}>0$ such that for any $\phi\in\mathcal{B}_{\sigma}(\phi_g)$, the following inequality holds
		\begin{align*}
			E(\phi)-E(\phi_g)\le \frac{1}{2\mu_{PL}}\left\|\nabla^{\mathcal{R}}_{\mathcal{P}}E(\phi)\right\|^2_{\mathcal{P}_{\phi}}.
		\end{align*}
		Then, $E$ satisfies the Morse--Bott condition around $\phi_g$.
	\end{lemma}
	\begin{proof}
		We first observe that, for $\sigma > 0$ sufficiently small, the Polyak--\L ojasiewicz inequality implies that any critical point in $\mathcal{B}_{\sigma}(\phi_g)$ has the same energy as $\phi_g$. Indeed, if  $ \phi_g' \in \mathcal{B}_{\sigma}(\phi_g) $  satisfies  $ \nabla_{\mathcal{P}}^{\mathcal{R}}E(\phi_g') = 0 $, then the Polyak--\L ojasiewicz inequality gives  $ E(\phi_g') - E(\phi_g) \le 0 $. Since  $ \phi_g $  is a local minimizer, we also have  $ E(\phi_g') \ge E(\phi_g) $, hence  $ E(\phi_g') = E(\phi_g) $. Therefore, for  $ \sigma > 0 $  sufficiently small, the set of critical points in  $ \mathcal{B}_{\sigma}(\phi_g) $  coincides with
		\begin{align*}
			\mathcal{S}_{\sigma}(\phi_g)=\left\{\phi\in\mathcal{B}_{\sigma}(\phi_g)\,\big|\,\nabla_{\mathcal{P}}^{\mathcal{R}}E(\phi)=0\right\}.
		\end{align*}
		Moreover, for any  $ \phi_g' \in \mathcal{S}_{\sigma}(\phi_g)$, since  $ \mathcal{B}_{\sigma}(\phi_g) $  is open, there exists  $ \sigma_{\phi_g'} > 0 $  such that 
		$
		\mathcal{B}_{\sigma_{\phi_g'}}(\phi_g') \subset \mathcal{B}_{\sigma}(\phi_g).
		$
		Consequently, for all  $ \phi \in \mathcal{B}_{\sigma_{\phi_g'}}(\phi_g') $, the Polyak--\L ojasiewicz inequality holds around  $ \phi_g' $ with the same constant $\mu_{PL}$.
		
		We now prove that, for  $ \sigma $  small enough,  $ \mathcal{S}_{\sigma}(\phi_g) $  is a finite-dimensional  $ C^1 $  embedded submanifold. For any  $ \phi_g' \in \mathcal{S}_{\sigma}(\phi_g)$, the kernel $K_{\phi_g^{'}}$ is finite-dimensional by Lemma \ref{PL-Co}. Let  $ \mathcal{R}_{\phi_g'}(v) $,  $ v \in T_{\phi_g'} \mathcal{M} $ , be a smooth local retraction satisfying  $ \mathcal{R}_{\phi_g'}(0) = \phi_g' $  and  $ \mathcal{R}'_{\phi_g'}(0) = I|_{T_{\phi_g'} \mathcal{M}} $. Then, by the inverse function theorem, there exists a sufficiently small neighborhood  $ \mathcal{B}_{\sigma_{\phi_g'}}(\phi_g') $  such that every  $ \phi \in \mathcal{B}_{\sigma_{\phi_g'}}(\phi_g') $  can be uniquely written as  $ \phi = \mathcal{R}_{\phi_g'}(v) $  for some small  $ v \in T_{\phi_g'} \mathcal{M} $. Define the pulled-back gradient map
		$
		G(v) := \nabla_{\mathcal{P}}^{\mathcal{R}} E(\mathcal{R}_{\phi_g'}(v)).
		$
		Then  $ G : T_{\phi_g'} \mathcal{M} \to T_{\mathcal{R}_{\phi_g'}(v)} \mathcal{M} $  is a  $ C^1 $  map with  $ G(0) = 0 $. Decompose  $ T_{\phi_g'} \mathcal{M} = K_{\phi_g'} \oplus R_{\phi_g'} $. Write  $ v = v_1 + v_2 $  with  $ v_1 \in K_{\phi_g'} $,  $ v_2 \in R_{\phi_g'} $, and define  $ G(v_1, v_2) := G(v_1 + v_2) $. Then
		$
		\partial_{v_2} G(0,0) = (\nabla_{\mathcal{P}}^{\mathcal{R}})^2 E(\phi_g')\big|_{R_{\phi_g'}}.
		$
		By Lemma \ref{PL-Co}, the Polyak--\L ojasiewicz inequality implies that this restriction is coercive
		\[
		\big( (\nabla_{\mathcal{P}}^{\mathcal{R}})^2 E(\phi_g') v_2, v_2 \big)_{\mathcal{P}_{\phi_g'}} \ge \mu_{PL} \|v_2\|_{\mathcal{P}_{\phi_g'}}^2, \quad \forall\; v_2 \in R_{\phi_g'},
		\]
		hence  $ \partial_{v_2} G(0,0) $  is an isomorphism from  $ R_{\phi_g'} $  onto $ R_{\phi_g'} $.
		By the implicit function theorem, there exists a  $ C^1 $  map  $ g $ from a neighborhood of  $  0  $  in  $  K_{\phi_g'}  $  into a neighborhood of $0$ in $R_{\phi_g'}$ such that $\forall\, v_1 \in \mathcal{U}_{\sigma}(0):=\{v\in K_{\phi_g'}\;|\;\|v\|_{H^1}<\sigma\}$,
		\[
		G(v_1 + g(v_1)) = 0\quad\text{and}\quad g'(0)=\left(\partial_{v_2} G(0,0)\right)^{-1}\partial_{v_1} G(0,0)=0,
		\]
		and every solution of  $ G(v) = 0 $  in a neighborhood of  $ 0 $  is of this form. Thus, we have constructed a local  $  C^1  $  chart for the set  $  \mathcal{S}_{\sigma}(\phi_g)  $  around any point  $  \phi_g'  $, given by 
		\begin{align*}
			\Phi(v_1) = \mathcal{R}_{\phi_g'}(v_1 + g(v_1))\quad\text{with}\quad \Phi'(0) = I|_{K_{\phi_g'}}.
		\end{align*}
		By the inverse function theorem, this is a local diffeomorphism from a neighborhood of  $  0  $  in  $  K_{\phi_g'}  $  onto a neighborhood of  $  \phi_g'  $  in  $  \mathcal{S}_{\sigma}(\phi_g)  $.
		
		It remains to show that 
		$
		\dim K_{\phi_g'} \equiv \dim K_{\phi_g}, \quad \forall\, \phi_g' \in \mathcal{S}_{\sigma}(\phi_g).
		$
		Using \eqref{Uniform-co} in the appendix together with \textbf{(A4)}-$(i)$, we can deduce that for any  $  \phi_g' \in \mathcal{B}_{\sigma}(\phi_g)  $  with  $  \sigma > 0  $  sufficiently small, the Riemannian Hessian  $  (\nabla_{\mathcal{P}}^{\mathcal{R}})^2 E(\phi_g')  $  is coercive on the subspace  $  T_{\phi_g}^{\phi_g'}(R_{\phi_g}) \subset T_{\phi_g'}\mathcal{M}$, i.e.,
		\begin{align*}
			\Big( (\nabla_{\mathcal{P}}^{\mathcal{R}})^2 E(\phi_g') \, T_{\phi_g}^{\phi_g'} v,\, T_{\phi_g}^{\phi_g'} v \Big)_{\mathcal{P}_{\phi_g'}}
			&=\Big\langle (E''(\phi_g')-\lambda_{\phi_g'}\mathcal{I}) \, T_{\phi_g}^{\phi_g'} v,\, T_{\phi_g}^{\phi_g'} v \Big\rangle\\
			&\ge C \, \|T_{\phi_g}^{\phi_g'} v\|_{H^1}^2\ge C \, \|T_{\phi_g}^{\phi_g'} v\|_{\mathcal{P}_{\phi_g'}}^2, \quad \forall\, v \in R_{\phi_g}.
		\end{align*}
		Consequently,
		$
		K_{\phi_g'} \cap T_{\phi_g}^{\phi_g'}(R_{\phi_g}) = \{0\}.
		$
		This implies that if  $  \mathcal{J}_{\phi_g'} v_1 = \mathcal{J}_{\phi_g'} v_2  $  for  $  v_1, v_2 \in T_{\phi_g}^{\phi_g'}(R_{\phi_g})  $, then  $  v_1 - v_2 \in K_{\phi_g'} \cap T_{\phi_g}^{\phi_g'}(R_{\phi_g}) = \{0\}  $, and hence  $  v_1 = v_2  $. In other words,  $  \mathcal{J}_{\phi_g'}  $  is injective on  $  T_{\phi_g}^{\phi_g'}(R_{\phi_g})  $. 
		
		We consider the restriction of the projection $\mathcal{J}_{\phi_g'}$ to the tangent space $T_{\phi_g'}\mathcal{M}$, denoted by $\mathcal{J}_{\phi_g'}|_{T_{\phi_g'}\mathcal{M}} : T_{\phi_g'}\mathcal{M} \to R_{\phi_g'} \subset T_{\phi_g'}\mathcal{M}$. Then $\mathcal{J}_{\phi_g'}|_{T_{\phi_g'}\mathcal{M}}$ is the $\mathcal P_{\phi_g'}$-orthogonal projection onto $R_{\phi_g'}$ with kernel $K_{\phi_g'}$, hence
		\[
		\mathcal{J}_{\phi_g'}|_{T_{\phi_g'}\mathcal{M}} = I - \Pi_{K_{\phi_g'}} : T_{\phi_g'}\mathcal{M} \to T_{\phi_g'}\mathcal{M},
		\]
		where $\Pi_{K_{\phi_g'}}$ denotes the $\mathcal{P}_{\phi_g'}$-orthogonal projection onto $K_{\phi_g'}$. Since $\dim K_{\phi_g'}<\infty$, $\Pi_{K_{\phi_g'}}$ is a finite-rank operator, and therefore $\mathcal{J}_{\phi_g'}|_{T_{\phi_g'}\mathcal{M}}$ is a Fredholm operator of index zero.
		Consider the inclusion map
		$
		\iota_{\phi_g'}: T_{\phi_g}^{\phi_g'}(R_{\phi_g}) \hookrightarrow T_{\phi_g'}\mathcal{M}.
		$
		Since $T_{\phi_g}^{\phi_g'}(R_{\phi_g})$ is a closed subspace of finite codimension in $T_{\phi_g'}\mathcal{M}$, the map $\iota_{\phi_g'}$ is a Fredholm operator with
		$
		\operatorname{ind} \iota_{\phi_g'} = -\operatorname{codim}_{T_{\phi_g'}\mathcal{M}} T_{\phi_g}^{\phi_g'}(R_{\phi_g}).
		$
		Therefore, the restricted operator $\mathcal{J}_{\phi_g'}|_{T_{\phi_g}^{\phi_g'}(R_{\phi_g})} = \mathcal{J}_{\phi_g'}|_{T_{\phi_g'}\mathcal{M}}\circ \iota_{\phi_g'}$ is a composition of Fredholm operators and hence itself Fredholm. By the index formula for compositions (see \cite[Theorem~2.8 in Chapter~XVII, Section~2]{Lang1993}),
		\[
		\operatorname{ind}\;\mathcal{J}_{\phi_g'}|_{T_{\phi_g}^{\phi_g'}(R_{\phi_g})}
		= \operatorname{ind} \iota_{\phi_g'} + \operatorname{ind} \mathcal{J}_{\phi_g'}|_{T_{\phi_g'}\mathcal{M}}
		= -\operatorname{codim}_{T_{\phi_g'}\mathcal{M}} T_{\phi_g}^{\phi_g'}(R_{\phi_g}) + 0.
		\]
		On the other hand, by definition of the Fredholm index,
		\[
		\operatorname{ind}\;\mathcal{J}_{\phi_g'}|_{T_{\phi_g}^{\phi_g'}(R_{\phi_g})}
		= \dim \ker\;\mathcal{J}_{\phi_g'}|_{T_{\phi_g}^{\phi_g'}(R_{\phi_g})}
		- \operatorname{codim}_{T_{\phi_g'}\mathcal{M}} \mathcal{J}_{\phi_g'}\bigl(T_{\phi_g}^{\phi_g'}(R_{\phi_g})\bigr).
		\]
		Since $\mathcal{J}_{\phi_g'}$ is injective on $T_{\phi_g}^{\phi_g'}(R_{\phi_g})$, the kernel vanishes. Hence,
		\[
		\operatorname{ind}\;\mathcal{J}_{\phi_g'}|_{T_{\phi_g}^{\phi_g'}(R_{\phi_g})}
		= -\operatorname{codim}_{T_{\phi_g'}\mathcal{M}} \mathcal{J}_{\phi_g'}\bigl(T_{\phi_g}^{\phi_g'}(R_{\phi_g})\bigr).
		\]
		Comparing the two expressions for the index yields
		\[
		\operatorname{codim}_{T_{\phi_g'}\mathcal{M}} \mathcal{J}_{\phi_g'}\bigl(T_{\phi_g}^{\phi_g'}(R_{\phi_g})\bigr)
		= \operatorname{codim}_{T_{\phi_g'}\mathcal{M}} T_{\phi_g}^{\phi_g'}(R_{\phi_g}).
		\]
		Moreover, since the map $T_{\phi_g}^{\phi_g'}$ is an isomorphism, we have
		\[
		\operatorname{codim}_{T_{\phi_g'}\mathcal{M}} T_{\phi_g}^{\phi_g'}(R_{\phi_g})
		= \operatorname{codim}_{T_{\phi_g}\mathcal{M}} R_{\phi_g}
		= \dim K_{\phi_g}.
		\]
		Finally, because $R_{\phi_g'} \supset \mathcal{J}_{\phi_g'}\bigl(T_{\phi_g}^{\phi_g'}(R_{\phi_g})\bigr)$, it follows that
		\[
		\dim K_{\phi_g'} 
		= \operatorname{codim}_{T_{\phi_g'}\mathcal{M}} R_{\phi_g'}
		\leq \operatorname{codim}_{T_{\phi_g'}\mathcal{M}} \mathcal{J}_{\phi_g'}\bigl(T_{\phi_g}^{\phi_g'}(R_{\phi_g})\bigr)
		= \dim K_{\phi_g}.
		\]
		Reversing the roles of  $ \phi_g $  and  $ \phi_g' $  yields the opposite inequality  $ \dim K_{\phi_g} \le \dim K_{\phi_g'} $. Therefore,
		$
		\dim K_{\phi_g'} = \dim K_{\phi_g}, \forall\, \phi_g' \in \mathcal{S}_{\sigma}(\phi_g).
		$
		This shows that the set $\mathcal{S}_{\sigma}(\phi_g)$ is a $ C^1 $  embedded submanifold, and that the kernel of the Riemannian Hessian coincides with the tangent space to this submanifold at each point, i.e., the Morse--Bott condition holds.
	\end{proof}
	To sum up, for the Gross--Pitaevskii energy funcational, we establish that the Morse--Bott condition and the Polyak--\L ojasiewicz inequality are equivalent on Hilbert manifold $ \mathcal{M}$. To the best of our knowledge, \cite{2025Fast} is the first work to prove this equivalence in the setting of finite-dimensional manifolds.
	
	To prove {\bf Theorem \ref{Opt-convergence}}, we introduce the operator  \(\mathcal{G}_{\tau}(\phi^*_g): R_{\phi^*_g} \to R_{\phi^*_g}\) defined by
	\begin{align*}
		\mathcal{G}_{\tau}(\phi_g^*) &:=\mathcal{J}_{\phi_g^*} \left(I - \tau\mathcal{P}^{-1}_{\phi_g^*} \left(E''(\phi_g^*) - \lambda_{\phi_g^*} \mathcal{I}\right)\right)\Big|_{R_{\phi^*_g}}.
	\end{align*}
	The operator norm characterization of $\mathcal{G}_{\tau}(\phi_g^*)$ is given as follows.
	
	\begin{lemma}\label{G-spectrum}
		Let $E$ satisfy the Morse--Bott condition around $\phi_g$. Endow the normal space $R_{\phi_g^*}$ with the inner product $(\cdot,\cdot)_{\mathcal{P}_{\phi_g^*}}$. 
		Then, there exist $\sigma>0$ such that for all $\phi_g^*\in\mathcal{S}_{\sigma}(\phi_g)$,the operator norm of $\mathcal{G}_{\tau}(\phi_g^*)$ satisfies
		\[
		\|\mathcal{G}_{\tau}(\phi_g^*)\| = \max\{\,|1 - \tau \mu|,\; |1 - \tau L|\,\}.
		\]
	\end{lemma}
	\begin{proof}
		The operator $\mathcal{G}_{\tau}(\phi_g^*)$ is clearly bounded and linear on the Hilbert space $R_{\phi^*_g}$ endowed with the inner product $(\cdot,\cdot)_{\mathcal{P}_{\phi_g^*}}$. 
		Moreover, it is self-adjoint. Indeed, for any $u, v \in R_{\phi^*_g}$, we have
		\begin{eqnarray*}
			\lefteqn{ (\mathcal{G}_{\tau}(\phi_g^*) u, v)_{\mathcal{P}_{\phi_g^*}}
				\,\,=\,\, \big(\mathcal{J}_{\phi_g^*} \big( I - \tau \mathcal{P}_{\phi_g^*}^{-1} (E''(\phi_g^*) - \lambda_{\phi_g^*} \mathcal{I})  u, \, v \big)_{\mathcal{P}_{\phi_g^*}} } \\
			&=&\big(\big( I - \tau \mathcal{P}_{\phi_g^*}^{-1} (E''(\phi_g^*) - \lambda_{\phi_g^*} \mathcal{I})\big)  u, \, v \big)_{\mathcal{P}_{\phi_g^*}} 
			\,\, = \,\, \langle \mathcal{P}_{\phi_g^*}u, v \rangle - \tau \langle (E''(\phi_g^*) - \lambda_{\phi_g^*} \mathcal{I}) u, v \rangle \\
			&=&\langle \mathcal{P}_{\phi_g^*}v, u \rangle - \tau \langle (E''(\phi_g^*) - \lambda_{\phi_g^*} \mathcal{I}) v, u \rangle
			\,\,=\,\,(\mathcal{G}_{\tau}(\phi_g^*) v, u)_{\mathcal{P}_{\phi_g^*}}.
		\end{eqnarray*}
		Since $\mathcal{G}_{\tau}(\phi_g^*)$ is self-adjoint, it follows from the polarization identity that its operator norm admits the variational characterization 
		\[
		\|\mathcal{G}_{\tau}(\phi_g^*)\| 
		= \sup_{v \in R_{\phi^*_g} \setminus \{0\}} 
		\frac{|(\mathcal{G}_{\tau}(\phi_g^*) v, v)_{\mathcal{P}_{\phi_g^*}}|}{(v, v)_{\mathcal{P}_{\phi_g^*}}}.
		\]
		Using the definition of $\mathcal{G}_{\tau}(\phi_g^*)$ and the fact that $\mathcal{J}_{\phi_g^*}$ acts as the identity on $R_{\phi^*_g}$, we obtain
		\[
		(\mathcal{G}_{\tau}(\phi_g^*) v, v)_{\mathcal{P}_{\phi_g^*}} 
		= (v, v)_{\mathcal{P}_{\phi_g^*}} - \tau \langle (E''(\phi_g^*) - \lambda_{\phi_g^*} \mathcal{I}) v, v \rangle
		= (v, v)_{\mathcal{P}_{\phi_g^*}} \bigl( 1 - \tau R(v) \bigr),
		\]
		where $R(v) = \dfrac{\langle (E''(\phi_g^*) - \lambda_{\phi_g^*} \mathcal{I}) v, v \rangle}{\langle \mathcal{P}_{\phi_g^*} v, v \rangle}$.
		Hence,
		$
		\|\mathcal{G}_{\tau}(\phi_g^*)\| = \sup_{v \in R_{\phi^*_g} \setminus \{0\}} |1 - \tau R(v)|.
		$
		Then, by Property \ref{Property-P}-$(i)$,
		$
		\|\mathcal{G}_{\tau}(\phi_g^*)\| = \max\{ |1 - \tau \mu|, |1 - \tau L| \},
		$
		as claimed.
	\end{proof}
	
	\begin{lemma}\label{L-S}
		If the nonlinearity $f$ is real-analytic and $\phi_s \in \mathcal{M}$ is a critical point of $E$ (i.e., $\mathcal{H}_{\phi_s}\phi_s=\lambda_s\phi_s$, $\lambda_s=\langle\mathcal{H}_{\phi_s}\mathcal{I}\phi_s,\phi_s$), then the energy functional $E$ satisfies a refined \L ojasiewicz--Simon Riemannian gradient inequality in a neighborhood of $\phi_s$, i.e., there exist constants $\nu \in (0, \tfrac{1}{2}]$ and $\sigma>0$, with $\nu$ independent of $\mathcal{P}_{\phi}$, such that for any $\phi \in \mathcal{B}_{\sigma}(\phi_s)$, we have
		\[
		|E(\phi) - E(\phi_s)|^{1 - \nu} \leq C \, \left\|\nabla_{\mathcal{P}}^{\mathcal{R}}E(\phi)\right\|_{\mathcal{P}_{\phi}}.
		\]
	\end{lemma}
	\begin{proof}
		We first establish the \L ojasiewicz--Simon gradient inequality for the Riemannian gradient associated with the $H_0^1$-metric.  
		Since the nonlinearity $f$ is real-analytic, the energy functional $E$ is real-analytic on $H_0^1(\mathcal{D})$, and the constraint manifold $\mathcal{M}$ is also real-analytic.  
		Therefore, we may apply \cite[Corollary~5.2]{2020JFA} with $V = H = H_0^1(\mathcal{D})$, equipped with the $H_0^1$ inner product, provided that the Hessian operator 
		$\nabla_{H_0^1}^2 E(\phi) : H_0^1(\mathcal{D}) \to H_0^1(\mathcal{D})$ is Fredholm of index zero. A direct computation yields
		\begin{align*}
			\nabla_{H_0^1}^2 E(\phi) 
			&= \bigl(-\Delta\bigr)^{-1} E''(\phi) = \bigl(-\Delta\bigr)^{-1} \Bigl( -\tfrac{1}{2}\Delta + V - \Omega \mathcal{L}_z + f(\rho_\phi) + f'(\rho_\phi)\bigl(|\phi|^2 + \phi^2 \overline{\cdot}\,\bigr) \Bigr) \\
			&= \frac{1}{2}I + \bigl(-\Delta\bigr)^{-1} \Bigl( V - \Omega \mathcal{L}_z + f(\rho_\phi) + f'(\rho_\phi)\bigl(|\phi|^2 + \phi^2 \overline{\cdot}\,\bigr) \Bigr) =: \frac{1}{2}I + A_\phi.
		\end{align*}
		The operator $A_\phi$ is compact due to 
		\begin{align*}
			\|A_{\phi}v\|^2_{H_0^1}&=\left\langle\Bigl( V - \Omega \mathcal{L}_z + f(\rho_\phi) + f'(\rho_\phi)\bigl(|\phi|^2 + \phi^2 \overline{\cdot}\,\bigr) \Bigr)v,A_{\phi}v\right\rangle\\
			&\le C_{\phi}\left(\|v\|_{L^2}+\|v\|_{L^{6/(4-\theta)}}\right)\|A_{\phi}v\|_{H^1_0},\quad \forall\; v\in H_0^1(\mathcal{D}).
		\end{align*}
		Consequently, $\nabla_{H_0^1}^2 E(\phi)$ is a compact perturbation of the (scaled) identity operator and hence Fredholm of index zero. By \cite[Corollary~5.2]{2020JFA}, there exist constants $\nu \in (0, \tfrac{1}{2}]$ and $\sigma > 0$ such that for all $\phi \in \mathcal{B}_\sigma(\phi_s)$,
		\begin{align}\label{L-S-H}
			|E(\phi) - E(\phi_s)|^{1 - \nu} \leq C \, \bigl\| \nabla_{H_0^1}^{\mathcal{R}} E(\phi) \bigr\|_{H_0^1}.
		\end{align}
		
		Next, we show that the Riemannian gradient with respect to an arbitrary preconditioner $\mathcal{P}_\phi$ is equivalent to the one induced by the scaled $H_0^1$-metric (cf. \textbf{(A4)}-$(i)$). Using the relation between the two metrics, we compute
		\begin{align*}
			\bigl\| \nabla_{\mathcal{P}_\phi}^{\mathcal{R}} E(\phi) \bigr\|_{\mathcal{P}_\phi}^2 
			&= \bigl\langle \mathcal{H}_\phi \phi - \lambda_\phi \phi, \nabla_{\mathcal{P}_\phi}^{\mathcal{R}} E(\phi) \bigr\rangle = \bigl\langle (-\Delta) \nabla_{H_0^1}^{\mathcal{R}} E(\phi), \nabla_{\mathcal{P}_\phi}^{\mathcal{R}} E(\phi) \bigr\rangle \\
			&\leq C_{\phi} \, \bigl\| \nabla_{H_0^1}^{\mathcal{R}} E(\phi) \bigr\|_{H_0^1} \, \bigl\| \nabla_{\mathcal{P}_\phi}^{\mathcal{R}} E(\phi) \bigr\|_{\mathcal{P}_\phi},
		\end{align*}
		which implies
		$
		\bigl\| \nabla_{\mathcal{P}_\phi}^{\mathcal{R}} E(\phi) \bigr\|_{\mathcal{P}_\phi} \leq C_{\phi} \, \bigl\| \nabla_{H_0^1}^{\mathcal{R}} E(\phi) \bigr\|_{H_0^1}.
		$
		Conversely, this follows from the equivalence between the preconditioned metric $\mathcal{P}_\phi$ and the $H_0^1$-metric. Hence, the two gradients are norm-equivalent:
		$
		\bigl\| \nabla_{H_0^1}^{\mathcal{R}} E(\phi) \bigr\|_{H_0^1} \simeq \bigl\| \nabla_{\mathcal{P}_\phi}^{\mathcal{R}} E(\phi) \bigr\|_{\mathcal{P}_\phi},
		$
		with constants independent of $\phi$ in a neighborhood of $\phi_s$.
		Combining this equivalence with \eqref{L-S-H}, we obtain
		\[
		|E(\phi) - E(\phi_s)|^{1 - \nu} \leq C \, \bigl\| \nabla_{\mathcal{P}_\phi}^{\mathcal{R}} E(\phi) \bigr\|_{\mathcal{P}_\phi}.
		\]
		
		Moreover, since the above equivalence holds for any preconditioner $\mathcal{P}_\phi$, and the exponent $\nu$ arises solely from the analytic structure of $E$ and $\mathcal{M}$ (via the Fredholm property of the Hessian), it follows that $\nu$ is independent of the choice of preconditioner $\mathcal{P}_\phi$.
	\end{proof}
	
	With this, we are ready to prove the theorems.
	\subsection{Proof of main results}
	\label{proofs-main-results}
	
	\begin{proof}[Proof of {\bf Theorem \ref{finite-classified-thm}}] We first show that under the Morse--Bott condition, each symmetry orbit $\mathcal{S}_{\phi_g} \subset \mathcal{S}_g$ is separable. That is, there exists $\sigma > 0$ such that for any $\phi \in \mathcal{B}_{\sigma}(\mathcal{S}_{\phi_g})$, $E(\phi) = E_{\mathcal{S}_g}$, implies $\phi \in \mathcal{S}_{\phi_g}$, or equivalently $\mathcal{B}_{\sigma}(\mathcal{S}_{\phi_g}) \cap \mathcal{S}_g = \mathcal{S}_{\phi_g}$.
		
		By applying the local projection construction from \cite[Lemma 4.1]{2025OnBest}, there exists $\sigma>0$ such that for every $\phi\in \mathcal{B}_\sigma(\phi_g)\subset\mathcal{M}$, there exists $\phi_g' \in \mathcal{S}_{\sigma}(\phi_g)$ satisfying
		\begin{align*}
			(\phi - \phi_g',\, \ci \phi_g')_{L^2} = 0,\quad (\phi - \phi_g',\, \ci \mathcal{L}_z \phi_g')_{L^2} = 0,\quad\text{and}\quad \|\phi - \phi_g'\|_{H^1} \le C \|\phi - \phi_g\|_{H^1}.
		\end{align*}
		Moreover, we have the decomposition
		\[
		\phi - \phi_g' = \operatorname{Proj}^{L^2}_{\phi_g'}(\phi - \phi_g') - \frac{1}{2} \|\phi - \phi_g'\|_{L^2}^2 \, \phi_g'.
		\]
		Now, choose $\sigma > 0$ sufficiently small. For any $\phi \in \mathcal{B}_{\sigma}(\mathcal{S}_{\phi_g})$ and all sufficiently small $\varepsilon > 0$, a Taylor expansion of $E$ at $\phi_g'$ yields
		\begin{align*}
			E(\phi) - E(\phi_g)
			&= E(\phi) - E(\phi_g') \\
			&= \langle E'(\phi_g'),\, \phi - \phi_g' \rangle 
			+ \frac{1}{2} \langle E''(\phi_g')(\phi - \phi_g'),\, \phi - \phi_g' \rangle 
			+ o\big( \|\phi - \phi_g'\|_{H^1}^2 \big) \\
			&= \frac{1}{2} \big\langle \big(E''(\phi_g') - \lambda_{\phi_g'} \mathcal{I} \big)(\phi - \phi_g'),\, \phi - \phi_g' \big\rangle 
			+ o\big( \|\phi - \phi_g'\|_{H^1}^2 \big) \\
			&\ge \frac{\mu - \varepsilon}{2} \, (\phi - \phi_g',\, \phi - \phi_g')_{E''(\phi_g')},
		\end{align*}
		where $E''(\phi_g')$ is coercive by \textbf{Property \ref{Prop2}} together with the fact that $\lambda_{\phi_g'}>0$. Consequently, if $E(\phi) = E(\phi_g)$, the above inequality forces $\phi = \phi_g' \in \mathcal{S}_{\phi_g}$. This proves that each orbit $\mathcal{S}_{\phi_g}$ is separable in $\mathcal{S}_g$. 
		
		It remains to show that $\mathcal{S}_g$ is compact in $H_0^1(\mathcal{D})$. Clearly, $\mathcal{S}_g$ is bounded in $H_0^1(\mathcal{D})$. To prove compactness, it suffices to verify sequential compactness. Let $\{v^n\}_{n \in \mathbb{N}} \subset \mathcal{S}_g$ be an arbitrary sequence. By the boundedness of $\mathcal{S}_g$ and the Rellich--Kondrachov compact embedding $H_0^1(\mathcal{D}) \subset\subset L^p(\mathcal{D})$ for $1 \le p < 6$, there exist a subsequence (still denoted by $\{v^{n}\}$) and some $v^* \in H_0^1(\mathcal{D})$ such that
		\[
		v^{n} \rightharpoonup v^* \quad \text{weakly in } H_0^1(\mathcal{D}), \qquad
		v^{n} \to v^* \quad \text{strongly in } L^p(\mathcal{D}) \quad\text{ for } 1 \le p < 6.
		\]
		Since the nonlinearity $F(\rho_v)$ satisfies suitable growth conditions \textbf{(A3)} (i.e., $|F(|z|^2)| \le C(1 + |z|^{\theta+3})$ with $\theta < 3$), the sequence $\{F(\rho_{v^n})\}_{n\in\mathbb{N}}$ is uniformly integrable in $L^1$, and hence we obtain
		\[
		\lim_{n \to \infty} \int_{\mathcal{D}} F(\rho_{v^n}) \, \text{d}\bm{x} 
		= \int_{\mathcal{D}} F(\rho_{v^*}) \, \text{d}\bm{x}.
		\]
		Together with the fact that each $v^{n}\in \mathcal{S}_g$, it follows that
		\[
		\|v^{n}\|_{\mathcal{H}_0} \to \|v^*\|_{\mathcal{H}_0}\quad\text{with}\quad \mathcal{H}_0=\mathcal{H}_{\phi}-f(\rho_{\phi}),
		\]
		which implies strong convergence $v^{n} \to v^*$ in $H_0^1(\mathcal{D})$. Moreover, $v^*\in\mathcal{M}$ and $E(v^*)=E_{\mathcal{S}_g}$. Hence, $\mathcal{S}_g$ is sequentially compact and therefore compact in $H_0^1(\mathcal{D})$. 
		
		By the separable property of $\mathcal{S}_{\phi_{g}}$, there exists $\sigma_{\phi_{g}}>0$ such that $\mathcal{B}_{\sigma_{\phi_{g}}}(\mathcal{S}_{\phi_{g}}) \cap \mathcal{S}_g = \mathcal{S}_{\phi_g}$. The collection $\big\{\mathcal{B}_{\sigma_{\phi_{g}}}(\mathcal{S}_{\phi_{g}})\big\}_{\phi_{g} \in \mathcal{S}_g}$ forms an open cover of $\mathcal{S}_g$. Since $\mathcal{S}_g$ is compact, there exist finitely many orbits $\mathcal{S}_{\phi_{g,1}}, \dots, \mathcal{S}_{\phi_{g,N}}$ such that
		\[
		\mathcal{S}_g = \bigsqcup_{\ell=1}^N \mathcal{S}_{\phi_{g,\ell}},
		\]
		which establishes a well-defined classification of the global minimizers. 
	\end{proof}
	\begin{proof}[Proof of {\bf Theorem \ref{Opt-convergence}}] 
		By \textbf{Lemma~\ref{Equal-L}}, \textbf{Lemma~\ref{P-L-In}}, and~\cite[Theorem~4.2]{2025OnBest}, we have that for arbitrary $\varepsilon>0$, there exists $\sigma>0$, such that the sequence $\{\phi^n\}_{n\in\mathbb{N}}$ generated by the P-RG converges linearly for any initial point $\phi^0 \in \mathcal{B}_\sigma(\phi_g)$ and any step size $\tau \in \bigl(0,\, 2/(L+\varepsilon)\bigr)$. Moreover, it holds $\phi^*_g := \lim\limits_{n\to\infty} \phi^n \in \mathcal{S}_{\phi_{g}}$. The optimal local $Q$-linear convergence rate is established below.
		
		For all $u \in T_{\phi^*_g}\mathcal{M}$ and $v\in K_{\phi_g^*}$, the following orthogonality relations hold:
		\begin{align*}
			\bigl(\mathrm{Proj}_{\phi_g^*}^{\mathcal{P}_{\phi_g^*}} \mathcal{P}_{\phi_g^*}^{-1}(E''(\phi_g^*)-\lambda_{\phi_g^*}\mathcal{I})u,\; v\bigr)_{\mathcal{P}_{\phi_g^*}}
			&= \bigl(\mathrm{Proj}_{\phi_g^*}^{\mathcal{P}_{\phi_g^*}} \mathcal{P}_{\phi_g^*}^{-1}(E''(\phi_g^*)-\lambda_{\phi_g^*}\mathcal{I})v,\; u\bigr)_{\mathcal{P}_{\phi_g^*}} = 0.
		\end{align*}
		A local linearization of the P-RG iteration at $\phi_g^*$ yields
		\begin{align*}
			\phi^{n+1} - \phi_g^*
			&= \phi^n - \phi_g^* - \tau\, \mathrm{Proj}^{\mathcal{P}_{\phi_g^*}}_{\phi_g^*} \mathcal{P}_{\phi_g^*}^{-1}(E''(\phi_g^*)-\lambda_{\phi_g^*}\mathcal{I})(\phi^n - \phi_g^*) + o(\|\phi^n - \phi_g^*\|_{H^1}) \\
			&= (I - \mathcal{J}_{\phi_g^*})(\phi^n - \phi_g^*) + \mathcal{G}_\tau(\phi_g^*)\mathcal{J}_{\phi_g^*}(\phi^n - \phi_g^*) + o(\|\phi^n - \phi_g^*\|_{H^1}).
		\end{align*}
		From this expansion, the following decoupled asymptotic relations follow:
		\begin{align*}\left\{
			\begin{aligned}
				&(I - \mathcal{J}_{\phi_g^*})(\phi^{n+1} - \phi^n) = o(\|\phi^n - \phi_g^*\|_{H^1})\\
				&\mathcal{J}_{\phi_g^*}(\phi^{n+1} - \phi_g^*) = \mathcal{G}_\tau(\phi_g^*)\mathcal{J}_{\phi_g^*}(\phi^n - \phi_g^*) + o(\|\phi^n - \phi_g^*\|_{H^1})
			\end{aligned}\;.
			\right.
		\end{align*}
		Telescopic summation of the first relation gives
		\begin{align}
			\label{telescoping-identity}
			(I - \mathcal{J}_{\phi_g^*})(\phi^n - \phi_g^*)
			&= \sum_{k=n}^\infty (I - \mathcal{J}_{\phi_g^*})(\phi^k - \phi^{k+1}),
		\end{align}
		with $\| (I - \mathcal{J}_{\phi_g^*})(\phi^k - \phi^{k+1}) \|_{\mathcal P_{\phi_g^*}}
		\to 0$ as $\| \phi^k - \phi_g^* \|_{H^1}\to0$. In particular, for every $\sigma>0$ there exists $\delta_\sigma\to0$ as
		$\sigma\to0$ such that
		\begin{align*}
			\| (I - \mathcal{J}_{\phi_g^*})(\phi^k - \phi^{k+1}) \|_{\mathcal P_{\phi_g^*}} 
			\le \delta_\sigma \| \phi^k - \phi_g^* \|_{\mathcal P_{\phi_g^*}}
			\qquad\text{whenever }\| \phi^k - \phi_g^* \|_{H^1}\le\sigma.
		\end{align*}
		Hence, from the telescoping identity \eqref{telescoping-identity} and for $n$ large enough so that $\| \phi^k - \phi_g^* \|_{H^1}\le\sigma$ for all $k\ge n$, it follows that
		\begin{align}
			\label{estimate-telescoping-identity}
			\bigl\|(I-\mathcal J_{\phi_g^*})(\phi^n-\phi_g^*)\bigr\|_{\mathcal P_{\phi_g^*}}
			\le \delta_\sigma \sum_{k=n}^\infty \|\phi^k-\phi_g^*\|_{\mathcal P_{\phi_g^*}}.
		\end{align}
		
		Given the linear convergence of $\{\phi^n\}_{n\in\mathbb{N}}$, there exists $\rho \in (0,1)$ such that
		\[
		\|\phi^k - \phi_g^*\|_{\mathcal{P}_{\phi_g^*}} \le C \rho^{k-n}\|\phi^n - \phi_g^*\|_{\mathcal{P}_{\phi_g^*}}\quad\text{ for all}\;\; k\ge n.
		\]
		Consequently,
		\[
		{
			\bigl\|(I-\mathcal J_{\phi_g^*})(\phi^n-\phi_g^*)\bigr\|_{\mathcal P_{\phi_g^*}}
			\overset{\eqref{estimate-telescoping-identity}}{\le}  }
		\delta_{\sigma} \sum_{k=n}^\infty \|\phi^k - \phi_g^*\|_{\mathcal{P}_{\phi_g^*}}
		\le C \delta_{\sigma} \|\phi^n - \phi_g^*\|_{\mathcal{P}_{\phi_g^*}},
		\]
		where $ \delta_{\sigma}\to 0^+$ as $\sigma\to0^+$. It follows that
		\[
		(I - \mathcal{J}_{\phi_g^*})(\phi^n - \phi_g^*) = o(\|\phi^n - \phi_g^*\|_{H^1}),
		\]
		and hence
		\begin{align*}
			\phi^n - \phi_g^* &= \mathcal{J}_{\phi_g^*}(\phi^n - \phi_g^*) + o(\|\phi^n - \phi_g^*\|_{H^1})\\
			\mathcal{J}_{\phi_g^*}(\phi^{n+1} - \phi_g^*) &= \mathcal{G}_\tau(\phi_g^*)\mathcal{J}_{\phi_g^*}(\phi^n - \phi_g^*) + o\left(\|\mathcal{J}_{\phi_g^*}(\phi^n - \phi_g^*)\|_{H^1}\right)
		\end{align*}
		Therefore, the asymptotic convergence behavior of $\phi^n - \phi_g^*$ is entirely determined by its projected component $\mathcal{J}_{\phi_g^*}(\phi^n - \phi_g^*)$. By \textbf{Lemma~\ref{G-spectrum}}, the operator norm of $\mathcal{G}_\tau(\phi_g^*)$ restricted to $R_{\phi_g^*}$ is equal to $\max\{|1 - \tau \mu|,\, |1 - \tau L|\}$. For all $\phi^0 \in \mathcal{B}_\sigma(\phi_g)$ and $\tau \in (0, 2/(L+\varepsilon))$, the local $Q$-linear estimate holds
		\[
		\bigl\| \mathcal{J}_{\phi_g^*}(\phi^n - \phi_g^*) \bigr\|_{\mathcal{P}_{\phi_g^*}}
		\le  \bigl\| \mathcal{J}_{\phi_g^*}(\phi^{n-1} - \phi_g^*) \bigr\|_{\mathcal{P}_{\phi_g^*}}\bigl( \max\{ |1 - \tau \mu|,\, |1 - \tau L| \} + \varepsilon \bigr),
		\quad \forall\; n \ge 1.
		\]
		Combined with the decomposition $\phi^n - \phi_g^* = \mathcal{J}_{\phi_g^*}(\phi^n - \phi_g^*) + o(\|\phi^n - \phi_g^*\|_{H^1})$, this implies that
		\[
		\| \phi^n - \phi_g^* \|_{\mathcal{P}_{\phi_g^*}}
		\le \| \phi^{n-1} - \phi_g^* \|_{\mathcal{P}_{\phi_g^*}} \left( \max\{ |1 - \tau \mu|,\, |1 - \tau L| \} + \varepsilon \right),
		\quad \forall\; n \ge 1,
		\]
		for $\sigma$ sufficiently small (so the kernel component is negligible uniformly), possibly enlarging $\varepsilon$. In particular, the optimal contraction factor is attained when $\tau = 2/(L + \mu)$. In this case,
		\[
		\max\{ |1 - \tau \mu|,\, |1 - \tau L| \} = \frac{L - \mu}{L + \mu},
		\]
		and the optimal local $Q$-linear convergence rate is given by
		\[
		\| \phi^n - \phi_g^* \|_{\mathcal{P}_{\phi_g^*}}
		\le \| \phi^{n-1} - \phi_g^* \|_{\mathcal{P}_{\phi_g^*}} \left( \frac{L - \mu}{L + \mu} + \varepsilon \right),\quad\forall\;n\ge1.
		\]
	\end{proof}
	
	\begin{proof}[Proof of {\bf Theorem \ref{MB-QE}}]
		By \textbf{Lemma~\ref{Equal-L}}, \textbf{Lemma~\ref{P-L-In}}, and~\cite[Lemma~4.3]{2025OnBest}, the sufficiency is immediate. We now prove the necessity. By Taylor expansion, we have
		\begin{align*}
			E(\phi^{1}) 
			= E(\phi^0) - \tau \|d_0\|_{\mathcal{P}_{\phi^0}}^2 
			+ \frac{\tau^2}{2} \big\langle (E''(\phi^0) - \lambda_{\phi^0} \mathcal{I}) d_0, d_0 \big\rangle 
			+ o(\tau^2 \| d_0\|_{H^1}^2).
		\end{align*}
		Note that since $d_0=\nabla_{\mathcal{P}}^{\mathcal{R}}E(\phi^0)$, there exists a constant $C_d>0$ such that $\|d_0\|_{H^1} \le C_d\,\|\phi^0-\phi_g\|_{H^1}$. In particular, for $\phi^0\in\mathcal B_\sigma(\phi_g)$, we have $\|d_0\|_{H^1}\le C\,\sigma$. We obtain
		\begin{align*}
			E(\phi^0) 
			= E(\phi^{1}) + \tau \| d_0 \|_{\mathcal{P}_{\phi^0}}^2 
			- \frac{\tau^2}{2} \big\langle (E''(\phi^0) - \lambda_{\phi^0} \mathcal{I}) d_0, d_0 \big\rangle 
			+ o(\|d_0\|_{H^1}^2).
		\end{align*}
		By the continuity of  $ E''(\phi) - \lambda_{\phi} \mathcal{I} $ ,  $ \mathcal{P}_{\phi} $, and  $ \operatorname{Proj}^{\mathcal{P}_{\phi}}_{\phi} $  with respect to  $ \phi $, together with \textbf{Lemma~\ref{Equal-L}}, we get for all $d \in T_{\phi}\mathcal{M}$ (with small $\|d\|_{H^1}$) that
		\[
		\big| \big\langle (E''(\phi) - \lambda_{\phi} \mathcal{I}) d, d \big\rangle \big| 
		\le L\|d\|_{\mathcal{P}_{\phi}}^2+ o(\|d\|_{H^1}^2).
		\]
		Consequently, for every sufficiently small $\varepsilon>0$, there exists sufficiently small $\sigma > 0$ such that for all  $ \phi\in\mathcal{B}_{\sigma}(\phi_g) $,
		\begin{align*}
			E(\phi^0) - E(\phi^1) 
			&\le \left( \tau + \frac{\tau^2}{2}(L + \varepsilon) \right) \|d_0\|_{\mathcal{P}_{\phi^0}}^2.
		\end{align*}
		On the other hand, by the assumed  $Q$-linear convergence of the energy sequence, 
		we have
		\begin{align*}
			E(\phi^0) - E(\phi^1) 
			&= \big( E(\phi^0) - E(\phi_g) \big) - \big( E(\phi^1) - E(\phi_g) \big) \\
			&\ge \big( 1 - \rho- \varepsilon \big) \big( E(\phi^0) - E(\phi_g) \big).
		\end{align*}
		Combining the two estimates yields
		\begin{align*}
			E(\phi^0) - E(\phi_g) 
			&\le \frac{ \tau + \frac{\tau^2}{2}(L + \varepsilon) }{ 1 - \rho -\varepsilon} \, \|d_0\|_{\mathcal{P}_{\phi^0}}^2,
		\end{align*}
		which establishes the Polyak--\L ojasiewicz inequality. This, together with \textbf{Lemma~\ref{PL-MB}}, completes the proof of necessity. 
	\end{proof}
	\begin{proof}[Proof of {\bf Theorem \ref{MB-CL}}]
		The sufficiency follows immediately from \textbf{Theorem~\ref{finite-classified-thm}} and \textbf{Theorem~\ref{Opt-convergence}}. We now prove the necessity. 
		
		Assume that the well-defined classification holds and that the iterates generated by the P-RG  converge linearly to the ground state manifold $\mathcal{S}_g$.
		Then, for any ground state $\phi_g \in \mathcal{S}_g$, the sequence converges linearly in a neighborhood of $\phi_g$. 
		As in the beginning of the proof of \textbf{Theorem~\ref{Opt-convergence}}, one obtains
		\begin{align*}
			\phi^0 - \phi_g^* 
			= \mathcal{J}_{\phi_g^*}(\phi^0 - \phi_g^*) + o(\|\phi^0 - \phi_g^*\|_{H^1}) \quad \text{as } \sigma \to 0^+,
		\end{align*}
		here $\phi^0 \in \mathcal{B}_\sigma(\phi_g)$ and the limit point $\phi_g^* \in \mathcal{S}_{\phi_g}$ satisfies $\|\phi^0 - \phi_g'\|_{H^1} \le C\sigma$. 
		We remark that this expansion is identical to the one in \eqref{kernel-vanish}.
		Let $I_{\alpha}^{\beta}$ denote the linear group action corresponding to phase shifts and rotations. By the invariance of the energy functional $E$ under the action of the linear group, for all $\phi\in\mathcal{S}_{\phi_g}$, we have the equivariance of the tangent bundle and the second variation
		\begin{align*}
			I_{\alpha}^{\beta} \big( T_{\phi}\mathcal{M} \big) = T_{I_{\alpha}^{\beta}\phi}\mathcal{M}\quad\text{and}\quad 
			\big\langle (E''(I_{\alpha}^{\beta}\phi) - \lambda_{I_{\alpha}^{\beta}\phi}\mathcal{I}) I_{\alpha}^{\beta}v,\, I_{\alpha}^{\beta}v \big\rangle 
			= \big\langle (E''(\phi) - \lambda_{\phi}\mathcal{I}) v,\, v \big\rangle.
		\end{align*}
		Moreover, by \textbf{(A4)} and the fact that $\mathcal{S}_{\phi_g}$ is bounded in $H^1$, the norms
		$\|\cdot\|_{\mathcal{P}_\phi}$ and $\|\cdot\|_{H^1}$ are uniformly equivalent for $\phi\in\mathcal{S}_{\phi_g}$
		(and hence also in a sufficiently small neighborhood of the orbit). Therefore, coercivity of the
		quadratic form in $H^1$ transfers to a uniform coercivity bound in the $\mathcal{P}_\phi$-norm.
		Consequently, the Riemannian Hessian of $E$ is uniformly coercive on the subspace $R_{\phi}$ over the entire orbit $\mathcal{S}_{\phi_g}$. 
		By repeating the Taylor-expansion argument used in \textbf{Lemma~\ref{P-L-In}},
		together with the assumed linear convergence of $\{\phi^n\}_{n\in \mathbb{N}}$,
		we obtain the Polyak--\L{}ojasiewicz inequality in a neighborhood of $\phi_g$.
		Since the Polyak--\L{}ojasiewicz inequality is equivalent to the Morse--Bott
		condition, the necessity is established.
	\end{proof}
	
	\begin{proof}[Proof of {\bf Theorem \ref{sublinear-convergence}}]
		By \cite[Lemma~4.3]{2025OnBest}, which provides the upper bound $E''(\phi_g) - \lambda_{\phi_g} \mathcal{I}$ in the preconditioner-norm and does not require the Morse--Bott assumption for this estimate, together with \textbf{Lemma~\ref{Equal-L}}, we obtain for any critical point $\phi_s'$ and every sufficiently small $\varepsilon > 0$, there exists $\sigma > 0$ such that for all $\phi \in \mathcal{B}_{\sigma}(\phi_s')$, the local energy dissipation satisfies
		\begin{equation}\label{eq:descent}
			E(\phi^{n+1}) - E(\phi^n) \le - C_{\tau} \left\|\nabla_{\mathcal{P}}^{\mathcal{R}} E(\phi^n)\right\|^2_{\mathcal{P}_{\phi^n}},
		\end{equation}
		where $C_{\tau} = \tau - \frac{\tau^2}{2}(L + \varepsilon) > 0$ for all $\tau \in \bigl(0,\, 2/(L + \varepsilon)\bigr)$.
		
		Moreover, by \textbf{Proposition~\ref{Property-P}}-$(iv)$, we have
		\begin{align*}
			\|\phi^{n+1} - \phi^n\|_{\mathcal{P}_{\phi^n}}
			= \left\|\nabla_{\mathcal{P}}^{\mathcal{R}} E(\phi^n)\right\|_{\mathcal{P}_{\phi^n}} 
			+ o\!\left(\left\|\nabla_{\mathcal{P}}^{\mathcal{R}} E(\phi^n)\right\|_{\mathcal{P}_{\phi^n}}\right).
		\end{align*}
		Combining this with the local energy dissipation estimate \eqref{eq:descent}, we obtain
		\begin{align*}
			E(\phi^n) - E(\phi^{n+1})
			\ge C_{\tau} \left\|\nabla_{\mathcal{P}}^{\mathcal{R}} E(\phi^n)\right\|_{\mathcal{P}_{\phi^n}} 
			\|\phi^{n+1} - \phi^n\|_{\mathcal{P}_{\phi^n}}.
		\end{align*}
		Hence, all the conditions of \cite[Theorem~2.1]{2022Exp} are satisfied, and in the absence of the Morse--Bott condition on the energy functional $E$, the P-RG admits the local sublinear convergence rate
		\[
		\|\phi^n - \phi_s\|_{H^1} \le C n^{-\frac{\nu}{1 - 2\nu}}, \quad  \nu \in(0,1/2) ,\quad\forall\;n\ge0,
		\]
		for some \(\phi_s \in \mathrm{Crit}\,(E)\).
	\end{proof}

	\section{Conclusion}\label{Sec6}
In this work, we established a geometric characterization of the ground-state manifold of the Gross–Pitaevskii energy functional based on the Morse–Bott condition. This structure provides a precise link between symmetry-induced degeneracy and the local convergence behavior of preconditioned Riemannian gradient methods. In particular, we identified the Morse–Bott property as the exact criterion separating linear from sublinear convergence and showed that it leads to a finite classification of symmetry-generated minimizers. These results highlight that the convergence behavior of optimization methods is governed by the intrinsic geometry of the critical set rather than by algorithmic details.

	 \bibliographystyle{siamplain}

	\appendix
	\setcounter{equation}{0}
	\section{Proof of $\mathcal{L}_z\phi\in H_0^1(\mathcal{D})$}\label{Appendix_Lz_phi_g}
	In this appendix, for the local minimizer $\phi_g$, we prove that $\mathcal{L}_z \phi_g \in H_0^1(\mathcal{D})$. To this end, we establish the following lemma.
	\begin{lemma}
		Let $\mathcal{D}\subset\mathbb{R}^d$ ($d=2,3$) be a bounded $C^{1,1}$-domain that is rotationally symmetric about the
		$z$-axis, i.e., $A_\beta \mathcal{D}=\mathcal{D}$ for all $\beta\in\mathbb{R}$. Then for every
		$u\in H^2(\mathcal{D})\cap H^1_0(\mathcal{D})$ we have
		\begin{align*}
			\mathcal{L}_z u \in H^1_0(\mathcal{D}).
		\end{align*}
		In particular, assumption {\bf (A3)} holds whenever $\phi_g\in H^2(\mathcal{D})\cap H^1_0(\mathcal{D})$.
	\end{lemma}

	\begin{proof}
		Since $u\in H^2(\mathcal{D})$ and $\mathcal{L}_z u = -\ci (x\partial_y u - y\partial_x u)$, we immediately have $\mathcal{L}_z u \in H^1(\mathcal{D})$.
		
		Because $A_\beta \mathcal{D}=\mathcal{D}$ for all $\beta$ and $A_\beta$ is a homeomorphism of $\mathbb{R}^d$, we have
		$A_\beta(\partial\mathcal{D})=\partial\mathcal{D}$. Hence, for every
		$\boldsymbol{x}\in\partial\mathcal{D}$ the curve $\beta\mapsto A_\beta\boldsymbol{x}$ lies in
		$\partial\mathcal{D}$. Differentiating with respect to $\beta$ at $\beta=0$ yields
		\begin{align*}
			0
			= \left.\frac{d}{d\beta} A_\beta \boldsymbol{x} \right|_{\beta=0}\cdot n(\boldsymbol{x})
			= (-y,\,x,\,0)\cdot n(\boldsymbol{x})
			\qquad \text{for a.e. } \boldsymbol{x}=(x,y,z)\in\partial\mathcal{D}.
		\end{align*}
		Consequently, the vector field $(-y,\,x,\,0)$ is tangential to $\partial\mathcal{D}$. Since
		\begin{align*}
			(x\partial_y - y\partial_x)u = (-y,\,x,\,0)\cdot\nabla u,
		\end{align*}
		this operator represents differentiation in a tangential direction along $\partial\mathcal{D}$.
		Because $\mathcal{D}$ is a $C^{1,1}$-domain and $u\in H^2(\mathcal{D})$, the trace
		$u|_{\partial\mathcal{D}}$ belongs to $H^{3/2}(\partial\mathcal{D})$, and tangential derivatives admit traces in
		$H^{1/2}(\partial\mathcal{D})$. Since $u\in H^1_0(\mathcal{D})$, we have
		$u|_{\partial\mathcal{D}}=0$, and therefore all tangential derivatives vanish on the boundary. In particular, $\big((x\partial_y - y\partial_x)u\big)\big|_{\partial\mathcal{D}} = 0$ in the sense of traces.
		Hence $\mathcal{L}_z u \in H^1(\mathcal{D})$ has vanishing trace on $\partial\mathcal{D}$, which shows that $\mathcal{L}_z u \in H^1_0(\mathcal{D})$.
	\end{proof}
	
	Finally, since the local minimizer $\phi_g$ satisfies the Euler--Lagrange equation
	\[
	\mathcal{H}_{\phi_g} \phi_g = \lambda_{\phi_g} \mathcal{I} \phi_g,
	\]
	standard elliptic regularity theory implies that $\phi_g \in H^2(\mathcal{D}) \cap H_0^1(\mathcal{D})$, provided $\mathcal{D}$ is $C^{1,1}$ and the nonlinearity satisfies the regularity assumptions in {\bf (A3)}. The precise regularity argument is e.g. ellobrated in \cite[Lem.~2.5]{2025Regularity}. Hence, by the lemma above, we conclude that
	$\mathcal{L}_z \phi_g \in H_0^1(\mathcal{D})$.

	\section{Proof of {\bf Property } \ref{Prop2}}\label{Appendix_Pro}
	\begin{proof}
		The coercivity at  $ \phi_g $  is identical to that established in \cite[Proposition~2.2]{2025OnBest}. We now prove uniform coercivity in a neighborhood of  $ \phi_g $. 
		
		Consider the restriction of the $L^2$-orthogonal projection operator (also known as vector transport):
		\begin{align}\label{Vec-Trans}
			T_{\phi_g}^{\phi_g'} := \operatorname{Proj}^{L^2}_{\phi_g'}\big|_{T_{\phi_g}\mathcal{M}}.
		\end{align}
		Since $\phi_g' \in \mathcal{B}_{\sigma}(\phi_g)$, we have
		\begin{align*}
			|(\phi_g',\phi_g)_{L^2}| = \left|1 - \frac{1}{2}\|\phi_g' - \phi_g\|_{L^2}^2\right| \ge C > 0
		\end{align*}
		for some constant $C$ independent of $\phi_g'$, provided $\sigma > 0$ is sufficiently small. To show that $T_{\phi_g}^{\phi_g'}$ is bijective, consider the equation
		\begin{align*}
			T_{\phi_g}^{\phi_g'}(v) = w, \quad v \in T_{\phi_g}\mathcal{M},\; w \in T_{\phi_g'}\mathcal{M}.
		\end{align*}
		This is equivalent to finding a scalar $x \in \mathbb{R}$ such that
		$
			v = w + x \phi_g' \in T_{\phi_g}\mathcal{M}.
		$
		Imposing the condition $(v, \phi_g)_{L^2} = 0$ yields
		$
			(w, \phi_g)_{L^2} + x (\phi_g', \phi_g)_{L^2} = 0,
		$
		which has a solution
		\begin{align*}
			x = -\frac{(w, \phi_g)_{L^2}}{(\phi_g', \phi_g)_{L^2}}.
		\end{align*}
		To prove uniqueness, suppose $v = w + x \phi_g'$ and $v_1 = w + x_1 \phi_g'$ both belong to $T_{\phi_g}\mathcal{M}$. Then $v - v_1 = (x - x_1)\phi_g' \in T_{\phi_g}\mathcal{M}$, so
		$
			(x - x_1)(\phi_g', \phi_g)_{L^2} = 0.
		$
		Since $(\phi_g', \phi_g)_{L^2} \neq 0$, it follows that $x = x_1$, and thus $v = v_1$. Therefore, $T_{\phi_g}^{\phi_g'}$ is bijective, and its inverse is given explicitly by
		\begin{align}\label{Invers-VT}
			(T_{\phi_g}^{\phi_g'})^{-1}(w) = w - \frac{(w, \phi_g)_{L^2}}{(\phi_g', \phi_g)_{L^2}}\, \phi_g', \quad \forall\, w \in T_{\phi_g'}\mathcal{M}.
		\end{align}
		Combined with the boundedness of $T_{\phi_g}^{\phi_g'}$ (its inverse is also bounded), it is a linear homeomorphism between $T_{\phi_g}\mathcal{M}$ and $T_{\phi_g'}\mathcal{M}$ for all $\phi_g' \in \mathcal{B}_\sigma(\phi_g)$ with $\sigma > 0$ sufficiently small.
		
		Using the continuity of $ E''(\phi)-\lambda_{\phi}\mathcal{I} $ and the  $L^2 $-projection  $ \operatorname{Proj}^{L^2}_{\phi} $  with respect to  $ \phi $, we deduce that  $  E''(\phi)-\lambda_{\phi}\mathcal{I}  $  is coercive on the closed subspace  $ T_{\phi_g}^{\phi_g'}(R_{\phi_g}) \subset T_{\phi_g'}\mathcal{M} $. More precisely, for some constant $\delta_{\sigma}$ with $\delta_{\sigma}\to 0^+$ for $\sigma\to 0^+$, we have 
		\begin{align}\label{Uniform-co}
			\nonumber\Big\langle (E''(\phi_g')-\lambda_{\phi_g'}\mathcal{I}) \, T_{\phi_g}^{\phi_g'} v,\, T_{\phi_g}^{\phi_g'} v \Big\rangle
			&\ge \Big\langle (E''(\phi_g)-\lambda_{\phi_g}\mathcal{I}) v, v \Big\rangle - \delta_{\sigma}\big\|T_{\phi_g}^{\phi_g'} v\big\|_{H^1}^2 \\
			&\ge C \, \|v\|_{H^1}^2 - \delta_{\sigma}\big\|T_{\phi_g}^{\phi_g'} v\big\|_{H^1}^2 \ge C \, \big\|T_{\phi_g}^{\phi_g'} v\big\|_{H^1}^2, \quad\forall\;   v \in R_{\phi_g},
		\end{align}
		for $\sigma$ sufficiently small. 
		This implies that
		$
		K_{\phi_g'} \cap T_{\phi_g}^{\phi_g'}(R_{\phi_g}) = \{0\}.
		$
		
		We now prove that the projection operator $\mathcal{J}_{\phi_g'}$ is a bijection from $ T_{\phi_g}^{\phi_g'}(R_{\phi_g})$ to $R_{\phi_g'}$. Since $T_{\phi_g}^{\phi_g'}$ is a linear homeomorphism and  $R_{\phi_g} \subset T_{\phi_g}\mathcal{M} $ is closed, it follows that $ T_{\phi_g}^{\phi_g'}(R_{\phi_g}) $  is closed in $T_{\phi_g'}\mathcal{M} $  and 
		\[
		\dim K_{\phi_g}=\operatorname{codim}_{T_{\phi_g}\mathcal{M}} R_{\phi_g}
		= \operatorname{codim}_{T_{\phi_g'}\mathcal{M}} T_{\phi_g}^{\phi_g'}(R_{\phi_g})=\dim(T_{\phi_g}^{\phi_g'}(R_{\phi_g}))^{\perp}_{\mathcal{P}_{\phi_g'}} .
		\]
		Because  $ K_{\phi_g'}$  is finite-dimensional and  $T_{\phi_g}^{\phi_g'}(R_{\phi_g})$ is closed,  $ T_{\phi_g}^{\phi_g'}(R_{\phi_g})+K_{\phi_g'} $  is a closed subspace of  $T_{\phi_g'}\mathcal{M}$.  
		Its orthogonal complement is
		\[
		(T_{\phi_g}^{\phi_g'}(R_{\phi_g})+K_{\phi_g'} )^{\perp}_{\mathcal{P}_{\phi_g'}} = 	(T_{\phi_g}^{\phi_g'}(R_{\phi_g}))^{\perp}_{\mathcal{P}_{\phi_g'}} \cap R_{\phi_g'}.
		\]
		We compute dimensions
		\begin{align*}
			\dim (T_{\phi_g}^{\phi_g'}(R_{\phi_g}))^{\perp}_{\mathcal{P}_{\phi_g'}} &= \dim((T_{\phi_g}^{\phi_g'}(R_{\phi_g}))^{\perp}_{\mathcal{P}_{\phi_g'}} \cap R_{\phi_g'}) + \dim((T_{\phi_g}^{\phi_g'}(R_{\phi_g}))^{\perp}_{\mathcal{P}_{\phi_g'}} \cap K_{\phi_g'}),\\
			\dim K_{\phi_g'} &= \dim(K_{\phi_g'} \cap T_{\phi_g}^{\phi_g'}(R_{\phi_g})) + \dim(K_{\phi_g'} \cap (T_{\phi_g}^{\phi_g'}(R_{\phi_g}))^{\perp}_{\mathcal{P}_{\phi_g'}}).
		\end{align*}
		By $K_{\phi_g'} \cap T_{\phi_g}^{\phi_g'}(R_{\phi_g}) = \{0\}$, we have $ \dim K_{\phi_g'} = \dim\big(K_{\phi_g'} \cap (T_{\phi_g}^{\phi_g'}(R_{\phi_g}))^{\perp}_{\mathcal{P}_{\phi_g'}}\big)$.  
		Substituting into the first equation and using  $\dim K_{\phi_g'} = \dim K_{\phi_g}=\dim(T_{\phi_g}^{\phi_g'}(R_{\phi_g}))^{\perp}_{\mathcal{P}_{\phi_g'}}$, we obtain
		\[
		\dim((T_{\phi_g}^{\phi_g'}(R_{\phi_g}))^{\perp}_{\mathcal{P}_{\phi_g'}} \cap R_{\phi_g'}) = 0.
		\]
		Therefore,  
		\[
		\overline{T_{\phi_g}^{\phi_g'}(R_{\phi_g}) + K_{\phi_g'}} = T_{\phi_g}^{\phi_g'}(R_{\phi_g}) + K_{\phi_g'}=T_{\phi_g'}\mathcal{M}.
		\]
		Then, for all $u\in R_{\phi_g'}$, there exist unique  $ v \in T_{\phi_g}^{\phi_g'}(R_{\phi_g})$  and  $ w \in K_{\phi_g'}$  such that $u=v+w$.
		Applying the orthogonal projection  $ \mathcal{J}_{\phi_g'} $ yields
		$
		u = \mathcal{J}_{\phi_g'}u = \mathcal{J}_{\phi_g'}v + \mathcal{J}_{\phi_g'}w =  \mathcal{J}_{\phi_g'}v.
		$
		Thus, $ \mathcal{J}_{\phi_g'} $ is a linear isomorphism from  $ T_{\phi_g}^{\phi_g'}(R_{\phi_g}) $  to  $ R_{\phi_g'} $, and consequently the composition operator $  \mathcal{J}_{\phi_g'} T_{\phi_g}^{\phi_g'}  $  is a linear isomorphism from  $  R_{\phi_g}  $  to  $  R_{\phi_g'}  $.
		Finally, the coercivity estimate transfers to  $ R_{\phi_g'} $: there exists a constant  $ C > 0 $, independent of  $ \phi_g' \in \mathcal{B}_\sigma(\phi_g)$, such that
		\begin{align*}
			\Big\langle (E''(\phi_g')-\lambda_{\phi_g'}\mathcal{I}) \mathcal{J}_{\phi_g'} T_{\phi_g}^{\phi_g'} v,\, \mathcal{J}_{\phi_g'}T_{\phi_g}^{\phi_g'} v \Big\rangle&=\Big\langle (E''(\phi_g')-\lambda_{\phi_g'}\mathcal{I}) \, T_{\phi_g}^{\phi_g'} v,\, T_{\phi_g}^{\phi_g'} v \Big\rangle\\
			&\ge C \, \big\|T_{\phi_g}^{\phi_g'} v\big\|_{\mathcal{P}_{\phi_g'}}^2\ge C \, \big\|\mathcal{J}_{\phi_g'}T_{\phi_g}^{\phi_g'} v\big\|_{H^1}^2 , \quad\forall\;   v \in R_{\phi_g}.
		\end{align*}
	\end{proof}

\end{document}